\documentclass[11pt,reqno]{amsart}
\usepackage{amsmath, amssymb, amsthm, mathrsfs, hyperref}
\usepackage[table]{xcolor}
\usepackage{stmaryrd} %\llbracket, \rrbracket
\usepackage{bbm} % \mathbbm
\usepackage{leftidx} %\ltrans, \leftidx
\usepackage{mathabx} %\widecheck
\usepackage{rotating} %\rotatebox
\usepackage{wasysym} %\lhd
\usepackage[mmddyy]{datetime}

\usepackage{arydshln}
\makeatletter
\renewcommand*\env@matrix[1][*\c@MaxMatrixCols c]{%
  \hskip -\arraycolsep
  \let\@ifnextchar\new@ifnextchar
  \array{#1}}
\makeatother
\usepackage{youngtab,young}

\newcommand{\iyng}[1]{{\tiny\Yvcentermath1 \yng(#1)}}
\newcommand{\iyoung}[1]{{\scriptsize\Yvcentermath1 \young(#1)}}

\usepackage{todonotes}
\usepackage{pgf, tikz, colortbl}
\usetikzlibrary{arrows, positioning, calc, chains, cd}
\tikzset{
	ch/.style={circle,draw,on chain,inner sep=2pt},
	chj/.style={ch,join},
	every path/.style={shorten >=4pt,shorten <=4pt}
	}
\newcommand{\dnode}[2][chj]{%
	\node[#1,label={below:#2}] (#1) {};}
\newcommand{\dnodenj}[1]{%
	\dnode[ch]{#1}}
\newcommand{\dydots}{%
	\node[chj,draw=none,inner sep=1pt] {\dots};}
\numberwithin{equation}{subsection}
\usepackage[margin = 0.9in]{geometry}
\newcommand{\bS}{\mbf S}

\newcommand{\bU}{\mbf U}

\newcommand{\bbH}{\mathbb{H}}
\newcommand{\bbS}{\mathbb{S}}

\newcommand{\A}{\mrm{A}}
\newcommand{\B}{\mrm{B}}

\newcommand{\R}{\mrm{R}}
\newcommand{\BA}[1]{    \begin{align}  #1  \end{align}}
\newcommand{\Ba}[1]{    \begin{array}   #1   \end{array}}
\newcommand{\Bc}[1]{     \begin{cases}   #1   \end{cases}}
\newcommand{\BM}[1]{    \begin{bmatrix}    #1    \end{bmatrix}}

\newcommand{\cC}{\mathcal{C}}

\newcommand{\cF}{\mathcal F}
\newcommand{\cH}{\mathcal{H}}

\newcommand{\cO}{\mathcal{O}}

\newcommand{\cT}{\mathcal{T}}

\newcommand{\CC}{\mathbb{C}}
\newcommand{\crl}{\curvearrowleft}
\newcommand{\crr}{\curvearrowright}

\newcommand{\ds}{\displaystyle}
\newcommand{\End}{\mrm{End}}
\newcommand{\Ext}{\mrm{Ext}}

\newcommand{\fgl}{\mathfrak{gl}}

\newcommand{\fs}{\mathfrak{s}}
\newcommand{\fsone}{\mathfrak{s}^{(1)}}
\newcommand{\fstwo}{\mathfrak{s}^{(2)}}
\newcommand{\fsl}{\mathfrak{sl}}
\newcommand{\ft}{\mathfrak{t}}
\newcommand{\ftone}{\mathfrak{t}^{(1)}}
\newcommand{\fttwo}{\mathfrak{t}^{(2)}}
\newcommand{\fu}{\mathfrak{u}}
\newcommand{\fuone}{\mathfrak{u}^{(1)}}
\newcommand{\futwo}{\mathfrak{u}^{(2)}}

\newcommand{\geo}{\mrm{geo}}
\newcommand{\alg}{\mrm{alg}}

\newcommand{\GL}{\mrm{GL}}
\newcommand{\Hom}{\mrm{Hom}}

\newcommand{\Irr}{\mrm{Irr}}

\newcommand{\lone}{\lambda^{(1)}}
\newcommand{\ltwo}{\lambda^{(2)}}

\newcommand{\inv}{^{-1}}

\newcommand{\mbf}{\mathbf}
\newcommand{\mrm}{\mathrm}
\newcommand{\Mod}{\mrm{Mod}}

\newcommand{\Mor}{\underset{\mrm{Mor}}{\simeq}}
\newcommand{\ind}{\mrm{ind}}

\newcommand{\MA}{M^\A_q}
\newcommand{\MB}{M^\B_{Q,q}}

\newcommand{\NN}{\mathbb{N}}
\newcommand{\nhc}{\lceil {\textstyle\frac{n}{2}}\rceil} %n half ceil
\newcommand{\nhf}{\lfloor {\textstyle\frac{n}{2}}\rfloor} %n half floor
\newcommand{\otw}{\textup{otherwise}}

\newcommand{\QQ}{\mathbb{Q}}

\newcommand{\tif}{\textup{if }}

\newcommand{\tr}{\vartriangleright}

\newcommand{\ZZ}{\mathbb{Z}}
\newcommand{\gt}{>}
\newcommand{\lt}{<}

\renewcommand{\=}[1]{\overline{#1}}

\newenvironment{enu}{\begin{enumerate}}{\end{enumerate}}
\usepackage{enumerate}
\newenvironment{eq}{\begin{equation}}{\end{equation}}

\theoremstyle{definition}
\newtheorem{Def}{Definition}[subsection]
\newtheorem{exa}[Def]{Example}
\newtheorem{rmk}[Def]{Remark}

\theoremstyle{plain}
\newtheorem{prop}[Def]{Proposition}
\newtheorem{thm}[Def]{Theorem}
\newtheorem{lemma}[Def]{Lemma}
\newtheorem{lem}[Def]{Lemma}
\newtheorem{cor}[Def]{Corollary}
\newtheorem{conj}[Def]{Conjecture}

\begin{document}

\title{On $q$-Schur algebras corresponding to Hecke algebras of type B}

\begin{abstract} In this paper the authors investigate the $q$-Schur algebras of type B that were  
constructed earlier using coideal subalgebras for the quantum group of type A. The authors 
present a coordinate algebra type construction that allows us to realize these $q$-Schur algebras 
as the duals of the $d$th graded components of certain graded coalgebras. Under suitable conditions 
an isomorphism theorem is proved that demonstrates that the representation theory reduces to the $q$-Schur algebra of 
type A. This enables the authors to address the questions of cellularity, quasi-hereditariness and representation type 
of these algebras. Later it is shown that these algebras realize the 1-faithful quasi hereditary covers of the Hecke algebras of type B. 
As a further consequence, the authors demonstrate that these algebras are Morita equivalent to Rouquier's finite-dimensional algebras that arise from the category ${\mathcal O}$ for rational Cherednik algebras for the Weyl group of type B. In particular, we have introduced a Schur-type functor that identifies the type B Knizhnik-Zamolodchikov functor.
\end{abstract}

\author{\sc Chun-Ju Lai}
\address{Department of Mathematics\\ University of Georgia \\
Athens\\ GA~30602, USA}
\email{cjlai@uga.edu}

\author{\sc Daniel K. Nakano}
\address{Department of Mathematics\\ University of Georgia \\
Athens\\ GA~30602, USA}
\thanks{Research of the second author was supported in part by
NSF grant DMS-1701768}
\email{nakano@math.uga.edu}

\author{\sc Ziqing Xiang}
\address{Department of Mathematics\\ University of Georgia \\
Athens\\ GA~30602, USA}
\thanks{The third author gratefully acknowledge funding support from the RTG in Algebraic Geometry, Algebra, and Number Theory at the University of Georgia, and from the NSF RTG grant DMS-1344994}
\email{ziqing@uga.edu}


\maketitle

%\tableofcontents

%=============================================
\section{Introduction}
%=============================================
\subsection{}
%=============================================
Schur-Weyl duality has played a prominent role in the representation theory of groups and algebras. The duality first appeared as method to 
connect the representation theory of the general linear group $\GL_{n}$ and the symmetric group $\Sigma_{d}$. This duality carries over 
naturally to the quantum setting by connecting the representation theory of quantum $\GL_{n}$ and the Hecke algebra $\cH_q(\Sigma_d)$ of the 
symmetric group $\Sigma_{d}$. 

Let $U_q(\fgl_n)$ be the Drinfeld-Jimbo quantum group. Jimbo showed in \cite{Jim86} that there is a Schur duality between 
$U_q(\fgl_n)$ and $\cH_q(\Sigma_d)$ on the $d$-fold tensor space of the natural representation $V$ of $U_q(\fgl_n)$. The $q$-Schur algebra of type A, 
$S_q^\A(n,d)$, is the centralizer algebra of the $\cH_q(\Sigma_d)$-action on $V^{\otimes d}$. 

It is well-known that the representation theory for $U_q(\fgl_n)$ is closely related to the representation theory for the quantum linear group $\GL_n$. The polynomial 
representations $\GL_n$ coincide with modules of $S_q^\A(n,d)$ with $d > 0$.  The relationship between objects are depicted as below:
\[
\Ba{{cccccccc}
K[M^\A_q(n)]^*&\hookleftarrow& 
U_q(\fgl_n)
\\
\downarrow&&
\downarrow
\\
K[M^\A_q(n)]_d^* &\simeq& 
S_q^\A(n,d) &\crr& V^{\otimes d}& \crl& \cH_q(\Sigma_d)
}
\]
The algebra $U_q(\fgl_n)$ embeds in the dual of the quantum coordinate algebra $K[M_q^\A]$; 
while $S_q^\A(n,d)$ can be realized as its $d$-th degree component. The reader is referred to \cite{PW91} for a thorough treatment of the subject. 

The Schur algebra $S_q^\A(n,d)$ and the Hecke algebra $\cH_q(\Sigma_d)$ are structurally related when $n\geq d$. 
\begin{itemize}
\item There exists an idempotent $e \in S_q^\A(n,d)$  such that $eS_q^\A(n,d)e \simeq \cH_q(\Sigma_d)$;
\item
An  idempotent yields the existence of Schur functor $\Mod(S_q^\A(n,d))\to \Mod(\cH_q(\Sigma_d))$;
\item 
$S_q^\A(n,d)$ is a (1-faithful) quasi-hereditary cover of $\cH_q(\Sigma_d)$\footnote{The algebra $S_q^\A(n,d)$ is 1-faithful under the conditions that $q$ is 
not a root of unity or if $q^2$ is a primitive $\ell$th root of unity then $\ell \geq 4$.}
%\item
%$U_q(\fgl_n)$ admits a stabilization realization from the family $\{S_q^\A(n,d)~|~ d \geq 0\}$.
\end{itemize}

\subsection{} Our paper aims to investigate the representation theory of the $q$-Schur algebras of type B that arises  
from the coideal subalgebras for the quantum group of type $\A$. 

We construct, for type $\B = \mrm{C}$, the following objects in the sense that all favorable properties mentioned in the previous section hold: %associated with the $q$-Schur algebras:
\[
\Ba{{cccccccc}
K[M^\B_{Q,q}(n)]^* &\hookleftarrow& U_{Q,q}^\B(n)
\\
\downarrow&&\downarrow
\\
K[M^\B_{Q,q}(n)]_d^* &\simeq&S_{Q,q}^\B(n,d) &\crr& V_\B^{\otimes d}& \crl& \cH_{Q,q}^\B(d)
}
\]
For our purposes it will be advantageous to work in more general setting with two parameters $q$ and $Q$, and construct the analogs $K[M^\B_{Q,q}(n)]$ of the quantum coordinate algebras. 
Then we prove that the $d$th degree component of $K[M^\B_{Q,q}(n)]^*$ is isomorphic to the type B $q$-Schur algebras.
The coordinate approach provide tools to study the representation theory for the algebra $K[M^\B_{Q,q}(n)]^*$ and for the $q$-Schur algebras simultaneously. The algebra $U_{Q,q}^\B(n)$, unlike $U_q(\fgl_n)$, does not have an obvious comultiplication.
%In fact, $U_q^\B(n)$ is a (right) coideal subalgebra of $\bU_n^\A = U_q(\fgl_n)$ in the sense that the comultiplication $\Delta$ on $\bU_n^\A$ sends 
%$\bU_n^\B$ to $\bU_n^\B \otimes \bU_n^\A$.
Therefore, its dual object, $K[M^\B_{Q,q}(n)]$, should be constructed as a coalgebra; while in the earlier situation $K[M^\A_q(n)]$ is a bialgebra. Our approach here differs from 
prior approaches to the subject for type $B$-Hecke algebra in that we can employ the action on the tensor space to realize the $q$-Schur algebra as a subcoalgebra (rather than simply defining the 
algebra as an endomorphism algebra on a direct sum of permutation modules).

In the second part of the paper an isomorphism theorem between the $q$-Schur algebras of type B and type A (under an invertibility condition) is 
established: 
\begin{equation}
S_{Q,q}^\B(n, d) \cong 
\Bc{
\ds\bigoplus_{i=0}^d S_q^\A(r, i) \otimes S_q^\A(r, d - i)
&\tif n=2r;
\\
\ds\bigoplus_{i=0}^d S_q^\A(r+1, i) \otimes S_q^\A(r, d - i)
&\tif n = 2r+1.
} 
\end{equation} 
One can view this as a ``lifting'' of the Morita equivalence (via the Schur functor) 
\begin{equation} \label{eq:HB}
\cH_{Q,q}^\B(d) \Mor  \prod_{i=0}^{d} \cH_q(\Sigma_i) \otimes \cH_q(\Sigma_{d-i}),
\end{equation} 
between Hecke algebras proved by Dipper-James \cite{DJ92}.
There are many cases when the Morita equivalence will hold, in particular when (i) $q$ is generic, (ii) $q$ is an odd root of unity, or (iii) $q$ is an (even) $\ell$th root of unity if $\ell > 4d$.

As a corollary of our isomorphism theorem, we obtain favorable properties for our coideal Schur algebras, see Section~\ref{sec:cell}--\ref{sec:qhcover}.  In particular, 
with the Morita equivalance we are able to show that $S_{Q,q}^\B(n, d)$ is a cellular algebra and quasi-hereditary. Moreover, in Section~\ref{sec:Rep}, we are able 
give a complete classification of the representation type of $S_{Q,q}^\B(n, d)$. 
In the following section (Section~\ref{sec:qhcover}), we are able to demonstrate that under 
suitable conditions, the Schur algebra $S_{Q,q}^\B(n, d)$ gives a concrete realization of quasi-hereditary one-cover for $\cH_{Q,q}^\B(d)$. 
The problem of concretely realizing these 
one-covers is in general an open problem for arbitrary Hecke algebras. 
Our results also exhibit how the representation theory of $S_{Q,q}^\B(n, d)$ 
is related to Rouquier's Schur-type algebras that arise from the category ${\mathcal O}$ for rational Cherednik algebras.  
In particular, we have introduced a Schur-type functor $F^\flat_{n,d}: \Mod(S^\B_{Q,q}(n,d)) \to \Mod(\cH^\B_{Q,q}(d))$ which is roughly defined by hitting an idempotent, and then identifies $F^\flat_{n,d}$ with the type B Knizhnik-Zamolodchikov functor, which is defined via monodromy.

\subsection{}
In the one-parameter case (i.e., $q=Q$), the algebra $U_q^\B(n)$ is the coideal subalgebra $\bU^\imath$ or $\bU^\jmath$ of $U_q(\fgl_n)$ in \cite{BW13} as a part of a double centralizer property (see also \cite{ES18} for a skew Howe duality viewpoint). 
The corresponding Schur algebras therein are denoted by $\bS^\imath$ or $\bS^\jmath$ to emphasize the fact that they 
arise from certain quantum symmetric pairs of type A III/IV associated with involutions $\imath$ or $\jmath$ on a Dynkin diagram of type $A_n$.
Namely, we have the identification below:
\[
U_q^\B(n) \equiv \Bc{
\bU^\jmath_r&\tif n = 2r+1;
\\
\bU^\imath_r&\tif n = 2r,
}
\quad
S_q^\B(n,d) \equiv \Bc{
\bS^\jmath(r,d)&\tif n = 2r+1;
\\
\bS^\imath(r,d)&\tif n = 2r.
}
\]
Note that the algebras $\bS^\jmath$ and its Schur duality are introduced first by Green in \cite{Gr97}. In \cite{BKLW} it is also developed a canonical basis theory for both Schur and coideal algebras.
For two parameters, a Schur duality for $U_{Q,q}^\B(n)$ is established in \cite{BWW18}; while the canonical basis theory can be found in \cite{LL18}.

To our knowledge, there is no general theory for finite-dimensional representations for the coideal subalgebras (see \cite{W17} for a classification for type AIII; also see \cite{Le17} for establishing their Cartan subalgebras for arbitrary type), and in some way our 
paper aims to establish results about  ``polynomial'' representations for $U_q^\B(n)$. 

There are other generalizations of the $q$-Schur duality for type B in literature. 
A comparison of the algebras regarding the aforementioned favorable properties will be given in Section~\ref{sec:SB}.
Since all these algebras are the centralizing partners of certain Hecke algebra actions, they are different from the ones appearing in the Schur duality (see \cite{H11}) for type B/C quantum groups, %and Birman-Murakami-Wenzl algebra, which is a quantization of the Brauer algebra.'
and are different from the coordinate algebras studied by Doty \cite{Do98}.

\noindent {\bf Acknowledgements.}
We thank Huanchen Bao, Valentin Buciumas, Hankyung Ko, Andrew Mathas, Stefan Kolb, Leonard Scott, Weiqiang Wang and Jieru Zhu for useful discussions.
The first author thanks the Academia Sinica for the support and hospitality during the completion of this project.
%=============================================
\section{Quantum coordinate (co)algebras}\label{sec:QCA}
%=============================================
\subsection{Quantum matrix spaces}
%=============================================
Let $K$ be a field  containing elements $q, Q$.
Denote the (quantum) commutators by
\eq
[A, B]_{x} = AB - xBA
\quad
(x \in K),
\quad
[A,B] = [A,B]_1.
\endeq

We define the (type A) quantum matrix spaces following \cite[\S3.5]{PW91} but with a shift on the index set as below:
\eq\label{eq:I}
I(n) = \Bc{
[-r, r]\cap\ZZ &\tif n = 2r+1;
\\
[-r, r]\cap\ZZ - \{0\} &\tif n = 2r.
}
\endeq
Let $\MA= M^\A_q(n)$ be the quantum analog of the space of $n\times n$ matrices indexed by $I(n)$, 
and let $K[\MA] = K[x_{ij}; i,j \in I(n)]/J^\A_q(n)$ be the associative $K$-algebra where
 $J^\A_q(n)$ is  the two-sided ideal of $ K[x_{ij}]$ generated by
\BA{
&{}[x_{ki}, x_{kj}]_{q\inv},
\quad &i>j
\\
&{}[x_{ki}, x_{li}]_{q\inv},
\quad &k>l
\\
&{}[x_{ki}, x_{lj}],
\quad &k>l, i<j
\\
&{}[x_{ki}, x_{lj}] - (q\inv - q) x_{li} x_{kj}.
\quad &k>l, i>j
}
% ========= RTT ==========
%Denote the $R$-matrix by
%\eq\label{def:R}
%R = q\sum_{i=1}^n E_{ii} \otimes E_{ii} + \sum_{i\neq j} E_{ii}\otimes E_{jj} + (q-q\inv) \sum_{i<j} E_{ij}\otimes E_{ji},
%\endeq
%where $E_{ij} \in \End(\CC^n)$ are matrix units.
%We remark that the relations above are equivalent to
%\eq
%q^{ \delta_{ik}} x_{ij}x_{kl} - q^{  \delta_{jl}} x_{kl}x_{ij} 
%= (q- q\inv)(\delta_{l < j} - \delta_{i < k}) x_{kj}x_{il},
%\endeq
%and hence is equivalent to the $RTT$ matrix equation
%\eq
%RX_1X_2 = X_2X_1 R,
%\endeq
%where $X_1 = (x_{ij}) \otimes \id, X_2^\pm = \id \otimes (x_{ij}) \in \End(\CC^n \otimes \CC^n)$.
% =======================
% ===== old type B ===========
%Let $J^\B_q(n)$ be the ideal of $K[M^A_q(\pm n)] $ generated by
%\BA{
%%&{}[x_{ki}, x_{kj}]_{q\inv},
%%\quad &(i>j)
%%\\
%%&{}[x_{ki}, x_{li}]_{q\inv},
%%\quad &(k>l)
%%\\
%%&{}[x_{ki}, x_{lj}],
%%\quad &(k>l, i<j)
%%\\
%%&{}[x_{ki}, x_{lj}] - (q\inv - q) x_{li} x_{kj},
%%\quad &(k>l, i>j)
%%\\
%&{}x_{ij}- x_{-i,-j},
%\quad &(i<0<j)\label{eq:B1}
%\\
%&{} x_{ij} - x_{-i,-j} - (q\inv - q) x_{-i,j}.
%&\quad(i,j < 0)\label{eq:B2}
%}
%=====================
The comultiplication on $K[\MA]$ is given by
\eq
\Delta:K[\MA] \to K[\MA] \otimes K[\MA],
\quad
x_{ij} \mapsto \sum_{k \in I(n)} x_{ik}\otimes x_{kj}.
\endeq
Let $V = V(n)$ be the $n$-dimensional vector space over $K$ with basis $\{v_i ~|~ i \in I(n)\}$. As a comodule  $V$ has a structure map
\eq
\tau_\A: V \to V \otimes K[\MA],
\quad
v_i \mapsto \sum_j v_j \otimes x_{ji}.
\endeq
For $\mu = (\mu_1, \ldots, \mu_d) \in I(n)^d$, set 
\eq 
v_\mu = v_{\mu_1} \otimes \ldots \otimes v_{\mu_d} \in V^{\otimes d}.
\endeq
It is easy to see that the set $\{ v_\mu ~|~ \mu \in I(n)^d\}$ forms a $K$-basis of the tensor space $V^{\otimes d}$. The structure map $\tau_\A$ induces a structure map
\eq
\tau^{\otimes d}_\A:
V^{\otimes d} \to V^{\otimes d} \otimes K[\MA],
\quad
v_\mu
\mapsto 
\sum_{\nu \in I(n)^d} v_\nu \otimes x_{\nu_1\mu_1} \ldots x_{\nu_d\mu_d}.
\endeq
In other words, the tensor space $V^{\otimes d}$ admits a $K[\MA]^*$-action defined by
\eq
K[\MA]^* \times V^{\otimes d} \to V^{\otimes d},
\quad
(f, v_\mu) 
\mapsto 
\sum_{\nu \in I(n)^d} f(x_{\nu_1\mu_1} \ldots x_{\nu_d\mu_d})  v_\nu  .
\endeq
%=============================================
\subsection{Hecke algebras of type B}
%=============================================
Let $\cH^\B = \cH^\B_{Q,q}(d)$ be the two-parameter Hecke algebra of type $\B$ over $K$ generated by $T_0, T_1, \ldots, T_{d-1}$ subject to the following relations:
\BA{
&
T_iT_{i+1}T_i=T_{i+1}T_iT_{i+1},
&
1\leq i \leq d-2,
\\
&
(T_0T_1)^2 = (T_1 T_0)^2,
\quad
T_iT_j=T_jT_i,
&
|i-j|>1,
\\
&
T_0^2 = (Q\inv -Q) T_0 + 1,
\quad
T_i^2 = (q\inv -q) T_i + 1,
&
1\leq i \leq d-1.
%(T_i-q^{-1})(T_i+q)=0 
}
That is, the corresponding Coxeter diagram is given as:
\[
\begin{tikzpicture}[start chain]
\dnode{$0$}
\dnodenj{$1$}
\dydots
\dnode{$d-1$}
\path (chain-1) -- node[anchor=mid] {\(=\joinrel=\joinrel=\)} (chain-2);
\end{tikzpicture}
\]

Let $W^\B(d)$ be the Weyl group of type B generated by $S = \{s_0, \ldots, s_{d-1}\}$. It is known that $\cH^\B_{Q,q}(d)$ has a $K$-basis $\{T_w ~|~ w \in W^\B(d)\}$, where $T_w = T_{i_1} \ldots T_{i_N}$ for any reduced expression $w = s_{i_1} \ldots s_{i_N}$.
The subalgebra of $\cH^\B_{Q,q}(d)$ generated by $T_1, T_2, \ldots, T_{d-1}$ is isomorphic to the Hecke algebra $\cH^\A_q(\Sigma_d)$ of the symmetric group $\Sigma_d$.
Let $\cH^\B_q(d)$ be the specialization of $\cH^\B_{Q,q}(d)$ at $Q=q$. 
%=============================================
\subsection{Type B Schur duality}\label{sec:SdB}
%=============================================
It is well-known that $V^{\otimes d}$ admits an $\cH^\B_{Q,q}(d)$-action (and hence an $\cH_q^\A(\Sigma_d)$-action) defined as follows. 
For $\mu = (\mu_i)_i \in I(n)^d$, $0\leq t \leq d-1$, let
\eq
\mu \cdot s_t = \Bc{
(\mu_1, \ldots, \mu_{t-1}, \mu_{t+1}, \mu_t, \mu_{t+2}, \ldots, \mu_d) 
&\tif t \neq 0;
\\
(-\mu_1, \mu_2, \ldots, \mu_d)
&\tif t = 0.
}
\endeq
For $1 \leq t \leq d-1$, the right $\cH^\B_{Q,q}(d)$-action on $V^{\otimes d}$ is defined:
\eq\label{eq:vTi}
v_\mu T_t = \Bc{
v_{\mu \cdot s_t}&\tif \mu_t< \mu_{t+1};
\\
q\inv v_{\mu \cdot s_t} &\tif \mu_t= \mu_{t+1};
\\
v_{\mu \cdot s_t} + (q\inv-q)v_\mu&\tif \mu_t > \mu_{t+1},
}
\quad
v_\mu T_0 = \Bc{
v_{\mu \cdot s_0}&\tif 0<\mu_{1};
\\
Q\inv v_{\mu \cdot s_0} &\tif 0= \mu_{1};
\\
v_{\mu \cdot s_0} + (Q\inv-Q) v_\mu &\tif 0 >\mu_{1}.
}
\endeq

The {\em $q$-Schur algebras} of type $\A$ (and $\B$, resp.) are denoted by
\eq\label{def:Sq}
S^\A = S^\A_q(n,d) = \End_{\cH^\A_q(d)}(V^{\otimes d}),
\quad
S^\B = S^\B_{Q,q}(n,d) = \End_{\cH^\B_{Q,q}(d)}(V^{\otimes d}).
\endeq
%\rmk
%The algebra $S^\B_{Q,q}(n,d)$ is referred as the hyperoctahedral Schur algebra in \cite{Gr97}.
%It is also known that $S_q^\B(n,d)$ has a $K$-basis indexed by the matrix set
%\eq
%\left\{A=(a_{ij}) \in \textup{Mat}_{[\pm n]\times[n]}(\NN) ~\middle|~ \sum_{i,j} a_{ij} = d\right\},
%\endeq
%and hence the dimension of $S_q^\B(n,d)$ is ${2n^2+d-1 \choose d}$. 
%\endrmk
We denote by $S_q^\B(n,d)$ the specialization of $S^\B_{Q,q}(n,d)$ at $Q=q$. 
It is known that $S_q^\B(n,d)$ admits a geometric realization (cf. \cite{BKLW}) as well as a Schur duality, which is compatible with the type A duality as follows:
\[
\Ba{{ccccccccc}
K[\MA(n)]^*&\twoheadrightarrow&K[\MA(n)]_d^*&\simeq&S_q^\A(n,d) &\crr& & \crl& \cH^\A(d)
\\
&&&&\cup&&V^{\otimes d}&&\cap
\\
&&&&S_q^\B(n,d)&\crr&&\crl&\cH_q^\B(d)
}
\]

%=============================================
\subsection{A coordinate coalgebra approach}
%=============================================

In this section, we aim to construct a coideal $J^\B_{Q, q}(n, d)$ of the coordinate bialgebra $K[M^\A_q(n)]_d$ such that $S^\B_{Q,q}(n,d)$ can be realized as the dual of the coordinate coalgebra 
\eq\label{eq:KBnd}
K[M^\B_{Q, q}(n)]_d = K[M^\A_q(n)]_d / J^\B_{Q, q}(n, d).
\endeq
\rmk
When it comes to comparing $S^\B_{Q,q}(n,d)$ with variants of $q$-Schur algebras type B (see Section~\ref{sec:SB}), we call $S^\B_{Q,q}(n,d)$ a {\it coideal }$q$-Schur algebra due to this nature.
\endrmk
For any $K$-subspace $J$ of $K[M^\A_q(n)]_d$, the $K[M^\A_{q}]_d$-comodule $V^{\otimes d}$ admits a $K[M^\A_{q}]_d/J$-comodule structure with structure map
\eq
\tau^{\otimes d}_J:
V^{\otimes d} \to V^{\otimes d} \otimes K[M^\A_{q}]_d/J,
\quad
v_\mu
\mapsto 
\sum_{\nu = (\nu_1, \ldots, \nu_d) \in I(n)^d} v_\nu \otimes (x_{\nu_1 \mu_1} \ldots x_{\nu_d \mu_d} +  J).
\endeq
We define a $K$-space $J^\B_{Q, q}(n, d)$ to be the intersection of all $K$-subspaces $J$ satisfying that 
\eq\label{eq:Tcomm}
(\tau_J^{\otimes d}(v_\mu)) T_0 = \tau_J^{\otimes d}(v_\mu T_0),
\quad
\textup{for all}
\quad
\mu \in I(n)^d.
\endeq
With $J^\B_{Q, q}(n, d)$, the linear space $K[M^\B_{Q, q}(n)]_d$ is well-defined as in \eqref{eq:KBnd}. We see from Proposition~\ref{prop:SBnd} that $K[M^\B_{Q, q}(n)]_d$ admits a coalgebra structure.
%-----------------------------------------------------------------------------------------------
\prop\label{prop:SBnd}
The $K$-space $K[\MB(n)]^*_d$ admits a $K$-algebra structure, and is isomorphic to the type B $q$-Schur algebra $S_{Q,q}^\B(n,d)$.
\endprop
%-----------------------------------------------------------------------------------------------
\proof
Let $\Psi$ be the $K$-algebra isomorphism $S^\A_q(n,d) \to K[M^\A_q]^*$, and hence 
\eq
\Psi(S^\B_{Q,q}(n,d)) = \{\phi \in K[M^\A_q(n, d)]^* ~|~ (\phi v_\mu)T_0 =  \phi (v_\mu T_0) \textup{ for all } \mu \in I(n)^d\}
\endeq
is a $K$-subalgebra of $K[M^\A_q(n, d)]^*$.
By the definition of $J^\B_{Q,q}(n,d)$, as linear spaces
\eq
K[M^\B_{Q, q}(n)]_d^*
=
\{\phi \in K[M^\A_q(n, d)]^* ~|~ \phi(r) = 0 \textup{ for all } r \in J^\B_{Q, q}(n, d)\}
=
\Psi(S^\B_{Q,q}(n,d)).
\endeq
Hence, $K[M^\B_{Q, q}(n)]_d^*$ is isomorphic to $S^\B_{Q,q}(n,d)$ as $K$-subalgebras of $K[M^\B_{Q, q}(n)]_d^*$. 
As a consequence, the space $J^\B_{Q,q}(n,d)$ is a coideal of $K[M^\B_{Q, q}(n)]_d^*$.
\endproof
%-----------------------------------------------------------------------------------------------

Let $J^\B_{Q,q}(n)$ be the union of the coideals $J^\B_{Q,q}(n,d)$ for all $d \in \NN$, and let
\eq
K[M^\B_{Q, q}(n)] = K[M^\A_q(n)]/ J^\B_{Q,q}(n).
\endeq
%-----------------------------------------------------------------------------------------------
\begin{cor}
The space $K[M^\B_{Q, q}(n)]$ of $K[M^\A_q(n)]$ is a quotient coalgebra.
\end{cor}
%-----------------------------------------------------------------------------------------------
\proof
It follows from that $J^\B_{Q, q}(n)$ is a coideal of $K[M^\A_q(n)]$ since its degree $d$ component $J^\B_{Q, q}(n, d)$ is a coideal of $K[M^\A_q(n)]_d$. 
\endproof
%-----------------------------------------------------------------------------------------------
Below we give a concrete realization of $J^\B_{Q, q}(n)$ as a right ideal.
 It is very important to observe that in general $J^\B_{Q, q}(n)$ is a right ideal and not a two-sided ideal, so $K[\MB(n)]$ is a coalgebra but {\em not} an algebra. 

%-----------------------------------------------------------------------------------------------
\prop
$J^\B_{Q,q}(n)$ is the right ideal of $K[M^\A_q(n)]$ generated by the following elements, for $i,j \in I(n)$.
\BA{
\label{eq:B1}
&x_{i,j}- x_{-i,-j},
&i<0<j
\\
\label{eq:B2}
&x_{i,j} - x_{-i,-j} - (Q\inv - Q) x_{-i,j},
&i,j < 0
\\
\label{eq:B3}
&x_{0,j} - Q\inv x_{0,-j},
&{j<0}
\\
\label{eq:B4}
&x_{i,0} - Q\inv x_{-i,0}.
&{i<0}
}
We remark that $I(2r)$ does not contain $0$ and hence $J^\B_{Q,q}(2r,d)$ is generated only by the elements of the form \eqref{eq:B1} -- \eqref{eq:B2}.
\endprop
%-----------------------------------------------------------------------------------------------
\proof
For a fixed $d\in\NN$, let $J$ be an arbitrary $K$-subspace of $K[M^\A_q(n)]_d$. For simplicity we write $\={x_{\mu\nu}} = x_{\mu\nu} + J$. For $i,j \in I(n)$ we write 
\[
\delta_{i<j} = \Bc{1&\tif i<j;\\ 0&\otw,}
\quad
\delta_{i>j} = \Bc{1&\tif i>j;\\ 0&\otw.}
\]
We first consider the case $d = 1$. 
For $i \in I(n)$, 
\eq
\begin{split}
(\tau_J^{\otimes 1}(v_i)) T_0 
&= \sum_{j \in I(n)} v_j T_0 \otimes \={x_{ji}}
\\
&= \sum_{j \in I(n)} 
\delta_{0j} Q^{-1} v_{-j} \otimes \={x_{ji}}
+ \delta_{0 \lt j} v_{-j} \otimes \={x_{ji}}
+ \delta_{0 \gt j} (v_{-j} + (Q^{-1} - Q) v_j) \otimes \={x_{ji}}
\\
&= \sum_{j \in I(n)} 
(\delta_{0j} Q^{-1} v_{j}% \otimes \={x_{-j,i}} 
+ \delta_{0 \lt -j} v_{j} %\otimes \={x_{-j, i}} 
+ \delta_{0 \gt -j} v_{j}) \otimes \={x_{-j, i}} 
+ \delta_{0 \gt j} (Q^{-1} - Q) v_{j}) \otimes \={ x_{ji}} 
\\
 &= \sum_{j \in I(n)} 
v_j \otimes \left(
    \delta_{0j} Q^{-1} \={x_{-j, i}} 
    + \delta_{0 \gt j} (\={x_{-j, i}} + (Q^{-1} - Q) \={x_{ji}}) 
    + \delta_{0 \lt j} \={x_{-j, i}} 
\right). 
\end{split}
\endeq
On the other hand,
\eq
\begin{split}
\tau_J^{\otimes 1}(v_i T_0)
&=  \tau_J^{\otimes 1}(\delta_{0i} Q^{-1} v_{-i} + \delta_{0 \lt i} v_{-i} + \delta_{0 \gt i} (v_{-i} + (Q^{-1} - Q) v_i)) 
\\
&=  \sum_{j \in I(n)} 
v_j \otimes \left( 
    \delta_{0i} Q^{-1} \={x_{j, -i}} 
    + \delta_{0 \lt i} \={ x_{j, -i}} 
    + \delta_{0 \gt i} \={ x_{j, -i}} 
    + (Q^{-1} - Q) \={ x_{ji}}
\right). 
\end{split}
\endeq
We then see that \eqref{eq:Tcomm} holds if and only if
 $J$ contains all the elements \eqref{eq:B1}--\eqref{eq:B4}.
Now, $J^\B_{Q,q}(n,1)$ is the linear space spanned by elements \eqref{eq:B1}--\eqref{eq:B4}
since it is the intersection of all the $J$'s satisfying \eqref{eq:Tcomm}.

For general $d$, since $T_0$ only acts on the first factor of $V^{\otimes d}$, the linear subspace $J^\B_{Q, q}(n, d)$ of $K[M^\A_q(n)]_d$ is $J^\B_{Q, q}(n, 1) \otimes K[M^\A_q(n)]_{d - 1}$. 
\endproof
%-----------------------------------------------------------------------------------------------

%=============================================
%\subsection{Representation theory}
%=============================================
Let $\tau_B = \tau_{J^\B_{Q,q}(n,d)}^{\otimes d}$.
We say that a right $K[M^\B_{Q,q}(n)]$-comodule $V$ is {\em homogeneous} of degree $d$ if all entries of its defining matrix lie in $K[M^\B_{Q,q}(n)]_d$, i.e., for a fixed basis $\{v_i\}$ of $V$, $\tau_\B(v_i) =\sum_j v_j \otimes a_{ij}$ for some $a_{ij} \in K[M^\B_{Q,q}(n)]_d$. 
%-----------------------------------------------------------------------------------------------
\begin{cor} For $d \geq 0$, the category of homogeneous right $K[M^\B_{Q,q}(n)]$-comodules of degree $d$ is equivalent to the category of left $S_{Q,q}^\B(n,d)$-modules.
\end{cor}
%-----------------------------------------------------------------------------------------------
%-----------------------------------------------------------------------------------------------
\subsection{A combinatorial realization of $S^\B_{Q,q}(n,d)$} \label{sec:comb}
%-----------------------------------------------------------------------------------------------
It is well-known that the algebra $S^\B_q(n,d)$ with equal parameters admits a geometric realization via isotropic partial flags (cf.  \cite{BKLW}).
This flag realization of $S^\B_q(n,d)$ admits a combinatorial/Hecke algebraic counterpart that generalizes to a two-parameter upgrade (cf.  \cite{LL18}), i.e., 
\eq\label{eq:phi}
S^\B_{Q,q}(n,d) = \bigoplus_{\lambda,\mu\in\Lambda^\B(n,d)}\Hom_{\cH^\B_{Q,q}}(x_\mu\cH^\B_{Q,q}, x_\lambda\cH^\B_{Q,q}),
\endeq
where
\eq\label{def:LBnd}
\Lambda^\B(n,d) = 
\Bc{
\left\{
\lambda = (\lambda_i)_{i\in I(n)} \in \NN^{n} ~\middle|~  \lambda_0 \in 1+2\ZZ, \lambda_{-i} = -\lambda_i, \sum_{i} \lambda_i = 2d+1
\right\}
&\tif n = 2r+1;
\\
\left\{
\lambda = (\lambda_i)_{i\in I(n)} \in \NN^{n} ~\middle|~  \lambda_{-i} = -\lambda_i, \sum_{i} \lambda_i = 2d
\right\}
&\tif n = 2r.
}
\endeq
Note that in \cite{LL18}, the set $\Lambda^\B(2r,d)$ is identified as a subset of $\Lambda^\B(2r+1,d)$ through the embedding
\[
(\lambda_i)_{i\in I(n)} \mapsto (\lambda_{-r}, \ldots, \lambda_{-1},1,\lambda_1,\ldots \lambda_r).
\]
For any $\lambda\in\Lambda^\B(n,d)$, let $W_\lambda$ be the parabolic subgroup of $W^\B$ generated by the set 
\eq
\Bc{
S - \{s_{\lambda_1}, s_{\lambda_1+\lambda_2}, \ldots, s_{\lambda_1 + \ldots +\lambda_{r-1}}\}
&\tif n = 2r;
\\
S - \{s_{\lfloor {\frac{\lambda_0}{2}}\rfloor}, s_{\lfloor {\frac{\lambda_0}{2}}\rfloor+\lambda_1}, \ldots, s_{\lfloor {\frac{\lambda_0}{2}}\rfloor +\lambda_1 + \ldots +\lambda_{r-1}}\}
&\tif n = 2r+1.
}
\endeq 
For any finite subset $X \subset W$, $\lambda, \mu\in\Lambda^\B(n,d)$ and a Weyl group element $g$, set
\eq\label{eq:x}
T_X = \sum_{w\in X} T_w,
\quad
T_{\lambda\mu}^g = T_{(W_\lambda)g(W_\mu)},
\quad
x_\lambda = T_{\lambda\lambda}^1 =T_{W_\lambda}.
\endeq
The right $\cH^\B_{Q,q}$-linear map below is well-defined:
\eq\label{def:phiB}
\phi_{\lambda\mu}^g: x_\mu \cH^\B_{Q,q} \to  x_\lambda \cH^\B_{Q,q},
\quad 
x_\mu \mapsto T^g_{\lambda\mu}.
\endeq
The maps $\phi_{\lambda\mu}^g$ with $\lambda, \mu \in \Lambda^\B(n,d)$, $g$ a minimal length double coset representative for $W_\lambda \backslash W^\B / W_\mu$ forms a linear basis for the algebra $S_{Q,q}^\B(n,d)$. 
The multiplication rule for $S^\B_{Q,q}(n,d)$ is given in \cite{LL18}, and it is rather involved in general. Here we only need the following facts:
%-----------------------------------------------------------------------------------------------
\begin{lem}\label{lem:SBmult}
Let $\lambda, \lambda', \mu, \mu' \in \Lambda^\B(n,d)$, and let $g, g'$ be minimal length double coset representatives for $W_\lambda \backslash W^\B / W_\mu$. Then
\enu[(a)]
\item
$\phi^g_{\lambda\mu} \phi^{g'}_{\lambda'\mu'} = 0$ unless $\mu= \lambda'$;
\item
$\phi^1_{\lambda\mu} \phi^{g}_{\mu\mu'} = \phi^g_{\lambda\mu'}= \phi^g_{\lambda\mu} \phi^{1}_{\mu\mu'}$.
\endenu
\end{lem}
%-----------------------------------------------------------------------------------------------
%=============================================
\subsection{Dimension of $q$-Schur algebras}
%=============================================
It is well-known that $S^\A_q(n,d)$ have several $K$-bases indexed by the set
$
\left\{
(a_{ij})_{ij} \in \NN^{I(n)^2}
~|~ 
\sum_{(i,j) \in I(n)^2} a_{i,j} = d
\right\}$,
and hence the dimension is given by
\eq\label{eq:dimSA}
\dim_K S^\A_q(n,d) =  {n^2+d-1\choose d}.
\endeq
In \cite[Lemma~2.2.1]{LL18} a dimension formula is obtained via several bases of $S_{Q,q}^\B(n,d)$ with the following index set:
\eq\label{eq:indexset}
\left\{
(a_{ij})_{ij} \in \NN^{I_-}
\middle| 
\sum_{(i,j) \in I_-} a_{i,\j} = d
\right\},
\quad
I_- = \Bc{
[-r,-1] \times I(n) &\tif n =2r;
\\
([-r,-1] \times I(n)) 
\cup
(\{0\} \times [-r,-1]) &\tif n =2r+1.
}
\endeq
That is, $I_- \subset I(n)^2$ correspond to the shaded region below:
\[
\begin{array}{cc}
\BM{[ccc:ccc]
\rowcolor{gray!40}
a_{-r,-r}&&&&&a_{-r,r}\\
\rowcolor{gray!40}
&\ddots&&&&\\
\rowcolor{gray!40}
&&a_{-1,-1}&a_{-1,1}&&\\
\hdashline
&&a_{1,-1}&a_{11}&&\\
&&&&\ddots&\\
a_{r,-r}&&&&&a_{rr}
}
&
\BM{[ccc:c:ccc]
\rowcolor{gray!40}
a_{-r,-r}&&&&&&a_{-r,r}\\
\rowcolor{gray!40}
&\ddots&&&&&\\
\hdashline
\cellcolor{gray!40}&\cellcolor{gray!40}&\cellcolor{gray!40}&\cellcolor{gray!40}a_{00}&&&\\
\hdashline
&&&&&\ddots&\\
a_{r,-r}&&&&&&a_{rr}\\
}
\\
\textup{if }n = 2r
&
\textup{if }n = 2r+1
\end{array}
\] 
Consequently, 
\eq\label{eq:dimSB}
\dim_K S_{Q,q}^\B(n,d) = {|I_-|+d-1\choose d}=
\Bc{
{2r^2 + d - 1 \choose d}&\tif n = 2r;
\\
{2r^2 + 2r + d\choose d}&\tif n = 2r+1.
}
\endeq
In the following we provide a concrete description for the $2$-dimensional algebra $S_{Q,q}^\B(2,1)$.

\prop\label{prop:SB(2,1)}
The algebra $S_{Q,q}^\B(2,1)$ is isomorphic to the type A Hecke algebra $\cH_{Q\inv}(\Sigma_2)$.
\endprop
\proof
The index set here is $I(2) = \{-1,1\}$.
The coalgebra $K[\MB(2)]_1$ has a $K$-basis $\{a = x_{-1,-1}, b= x_{-1,1}= x_{1,-1}\}$. Note that $x_{11} =  a +(Q-Q\inv)b$. The comultiplication is given by
\eq
\begin{split}
\Delta(a) &= \sum_{k =\pm1} x_{-1,k} \otimes x_{k, -1} = a\otimes a + b\otimes b,
\\
\Delta(b) &= b\otimes a + (  a +(Q-Q\inv)b)\otimes b = b\otimes a + a\otimes b + (Q-Q\inv) b\otimes b.
\end{split}
\endeq
Hence, the algebra structure of $S_{Q,q}^\B(2,1) = K[\MB(n)]_1^*$ has a basis $\{a^*, b^*\}$ such that
\eq
\begin{split}
&a^*a^*(a) = (a\otimes a)^*(\Delta(a)) = 1,
\quad
a^*a^*(b) = (a\otimes a)^*(\Delta(b)) = 0,
\\
&a^*b^*(a) = 0 =b^*a^*(a),
\quad
a^*b^*(b) = 1 =b^*a^*(b),
\\
&b^*b^*(a) = 1,
\quad
b^*b^*(b) = (Q-Q\inv),
\end{split}
\endeq
Therefore, the multiplication structure of $S_{Q,q}^\B(2,1)$ is given by
\eq
a^* a^* = a^*,
\quad 
a^* b^* = b^* = b^* a^*,
\quad
b^* b^* = (Q-Q\inv)b^* + a^*.
\endeq
\endproof
\rmk
We expect that the algebra $S_{Q,q}^\B(2,d)$ is isomorphic to $K[t]/\langle P_{d}(t)\rangle$ for some polynomial $P_d \in K[t]$, for $d \geq 1$.
\endrmk
%-----------------------------------------------------------------------------------------------
%\prop
%There is an algebra isomorphism $S_{Q,q}^\B(2,d) \simeq K[t]/\langle f_{1,d}(t)\rangle$.
%\endprop
%-----------------------------------------------------------------------------------------------
%\proof
%\red{TBA:
%description of $f_{1,d}(t)$. 
%}
%\endproof
%%-----------------------------------------------------------------------------------------------

%=============================================
\section{The Isomorphism Theorem}\label{sec:Iso}
%=============================================
The entire section is dedicated to the proof of an isomorphism theorem (Theorem~\ref{thm:SBA}) between the Schur algebras of type B and type A that is inspired by a Morita equivalence theorem due to Dipper and James \cite{DJ92}. 
%-----------------------------------------------------------------------------------------------
\subsection{The statement}
%-----------------------------------------------------------------------------------------------
We define a polynomial $f^\B_d \in K[Q,q]$ by
\eq 
f^\B_d(Q,q) = \prod_{i=1-d}^{d-1}(Q^{-2}+q^{2i}). 
\endeq
We remark that at the specialization $Q=q$, the polynomial $f^\B_d(Q, q)$ is invertible if (i) $q$ is generic, (ii) $q^2$ is an odd root of unity, or (iii) $q^2$ is a primitive (even) $\ell$th root of unity for $\ell > d$.
\begin{thm}\label{thm:SBA}
If $f^\B_d(Q, q)$ is invertible in the field $K$, then we have an isomorphism of $K$-algebras:
\eq\label{eq:Phi}
\Phi:S_{Q, q}^\B(n, d) \to \bigoplus_{i=0}^d S_q^\A(\lceil n / 2 \rceil, i) \otimes S_q^\A(\lfloor n / 2 \rfloor, d - i).
\endeq
\end{thm}

\exa
For $n=2, d=1$, Theorem~\ref{thm:SBA} gives the following isomorphism
\[
S_{Q, q}^\B(2, 1) \cong  
(S_q^\A(1,0) \otimes S_q^\A(1,1)) 
\oplus 
(S_q^\A(1,1) \otimes S_q^\A(1,0))
\cong K1_x \oplus K1_y, 
\]
where $1_x, 1_y$ are identities.
We recall basis $\{a^*, b^*\}$ of $S_{Q, q}^\B(2, 1)$ from Proposition~\ref{prop:SB(2,1)}.
The following assignments yield the desired isomorphism:
\eq\label{eq:iso21}
a^* \mapsto 1_x + 1_y,
\quad
b^* \mapsto - Q^{-1} 1_x + Q 1_y.
\endeq
We note that it remains an isomorphism if we replace $- Q^{-1} 1_x + Q 1_y$ in \eqref{eq:iso21} by $Q1_x - Q^{-1}1_y$.
\endexa

\subsection{Morita equivalence of Hecke algebras}
Following \cite{DJ92}, we define elements $u^\pm_i \in \cH^\B_{Q,q}(d)$, for $0\leq i \leq d$, by
\begin{equation} \label{eq:Frigostable}
u^+_i = \prod_{\ell=0}^{i-1} (T_{\ell}\ldots T_1T_0T_1 \ldots T_{\ell} + Q),
\quad
u^-_i = \prod_{\ell=0}^{i-1} (T_{\ell}\ldots T_1T_0T_1 \ldots T_{\ell} - Q\inv).
\end{equation}
It is understood that $u^+_0 = 1 = u^-_0$.
%For $a, b \in \NN$, we define an element $h_{a,b} = T_{w_{a,b}} \in \cH^\B_{Q,q}(d)$,
%where $w_{a,b}\in \Sigma_{a+b}$, in two-line notation, is
%\[ w_{a,b} = \left|
%\begin{array}{cccccc}1
%1 &\cdots& a &a+1& \cdots& a+b
%\\
%b+1& \cdots& b+a &1& \cdots &b
%\end{array}
%\right|. \]
For $a, b \in \NN$ such that $a+b =d$, we define an element $v_{a,b} $ by
\begin{equation} \label{eq:Reswarm}
v_{a, b} = u_b^- T_{w_{a, b}} u_a^+ \in \cH^\B_{Q,q}(d),
\end{equation}
where
$w_{a,b}\in \Sigma_{a+b}$, in two-line notation, is given by
\[ 
w_{a,b} = \left(
\begin{array}{cccccc}
1 &\cdots& a &a+1& \cdots& a+b
\\
b+1& \cdots& b+a &1& \cdots &b
\end{array}
\right). 
\]
Finally, when $f^\B_d(Q, q)$ is invertible, Dipper and James constructed an idempotent 
\eq\label{def:eab}
e_{a, b} = \tilde{z}_{b, a}^{-1} T_{w_{b, a}} v_{a, b},
\endeq
for $a + b = d$, where $\tilde{z}_{b, a}$ is some invertible element in $\cH_q(\Sigma_a \times \Sigma_b)$(see \cite[Definition~3.24]{DJ92}).
Below we recall some crucial lemmas used in \cite{DJ92}.

\begin{lemma} \label{lem:Balladical} Let $a, b \in \NN$ be such that $a+b = d$. Then:
\begin{enumerate}[(a)]
\item The elements $u^\pm_d$ lie in the center of $\cH^\B_{Q,q}(d)$,
\item For $a + b \gt d$, $u_b^- \cH^\B_{Q, q}(d) u_a^+ = 0$.
\item For $a + b = d$, $e_{a, b} \cH^\B_{Q, q}(d) e_{a, b} = e_{a, b} \cH_q(\Sigma_a \times \Sigma_b)$
and $e_{a, b}$ commutes with $\cH_q(\Sigma_a \times \Sigma_b)$,
\item For $a + d = d$, $e_{a, b} \cH^\B_{Q, q}(d) = v_{a, b} \cH^\B_{Q, q}(d)$,
\item There is a Morita equivalence 
\[
\cH_{Q,q}^\B(d) \Mor  \bigoplus_{i = 0}^{d} e_{i, d-i} \cH^\B_{Q, q}(d) e_{i, d-i}.
\]
\end{enumerate}
\end{lemma}

\subsection{The actions of $u_d^+$ and $u_d^-$}
Consider the following decompositions of $V$ into $K$-subspaces:
\[
V =  V_{\geq 0} \oplus V_{\lt 0} = V_{\gt 0} \oplus V_{\leq 0},
\]
where
\eq
V_{>0} = \bigoplus_{1\leq i \leq r} K v_i,
\quad
V_{\geq 0} = \Bc{
\bigoplus\limits_{0 \leq i \leq r} K v_i,
&\tif n = 2r+1
\\
V_{>0} &\tif n = 2r,
}
\endeq
\eq
V_{<0} = \bigoplus_{-r \leq i \leq -1} K v_i,
\quad
V_{\leq 0} = \Bc{
\bigoplus\limits_{-r \leq i \leq 0} K v_i,
&\tif n = 2r+1
\\
V_{<0} &\tif n = 2r.
}
\endeq
%\begin{align*}
%V_{\geq 0} := & \begin{cases}
%		K \left\langle v_0, v_1, \dots v_r \right\rangle, & n = 2 r + 1, \\
%		K \left\langle v_1, \dots v_r \right\rangle, & n = 2 r, \\
%	\end{cases} \\
%V_{\gt 0} := & K \left\langle v_1, \dots v_r \right\rangle, \\
%V_{\leq 0} := & \begin{cases}
%		K \left\langle v_{-r}, \dots v_{-1}, v_0 \right\rangle, & n = 2 r + 1, \\
%		K \left\langle v_{-r}, \dots v_{-1} \right\rangle, & n = 2 r, \\
%	\end{cases} \\
%V_{\lt 0} := & K \left\langle v_{-r}, \dots v_{-1} \right\rangle. \\
%V := & V_{\geq 0} \oplus V_{\lt 0} = V_{\gt 0} \oplus V_{\leq 0}.
%\end{align*}
Hence, one has  the following canonical isomorphisms:
\begin{equation} \label{eq:Gyroscopic}
%\begin{aligned}
%S^\B_{Q, q}(n, d) & = \End_{\cH^\B_{Q, q}(d)} V^{\otimes d}, \\
S^\A_q(\lceil n / 2\rceil, d)  \simeq \End_{\cH^\A_q(\Sigma_d)} (V_{\geq 0}^{\otimes d}), 
\quad
S^\A_q(\lfloor n / 2 \rfloor, d)  \simeq \End_{\cH^\A_q(\Sigma_d)} (V_{\lt 0}^{\otimes d}).
%\end{aligned}
\end{equation}

%\subsection{The action of $u^\pm_d$}

In the following, we introduce two new bases $\{w^+_I\}$ and $\{w^-_I\}$ for the tensor space to help us understand the $u^\pm_d$-action. First define some intermediate elements, for $0 \leq i \leq r, j \in \NN$: 
\eq
w_{i(j)}^+ = \begin{cases}
q^{-j} v_{-i} + Q v_i, & i \neq 0, \\
(q^{-2 j} Q^{-1} + Q) v_i, & i = 0, \\
\end{cases}
\quad 
\text{and} 
\quad 
w_{i(j)}^- = \begin{cases}
	q^{-j} v_{-i} - Q^{-1} v_i, & i \neq 0, \\
	0, & i = 0. \\
\end{cases}
\endeq
For a nondecreasing tuple $I = (i_1, \dots, i_d) \in ([0,r]\cap \ZZ)^d$, we further define elements $w^+_I$ and $w^-_I$ by 
\eq
w^+_{(i)} = w^+_{i(0)},
\quad
w^-_{(i)} = w^-_{i(0)},
\endeq
and then inductively (on $d$) as below:
\eq 
w^+_I = w^+_{(i_1, \dots, i_{d-1})} \otimes w^+_{i_d(j)}, 
\quad 
w^-_I = w^-_{(i_1, \dots, i_{d-1})} \otimes w^-_{i_d(j)}, 
\quad \text{where} \quad j = \max \{k : i_{d - k} = i_d\} 
\endeq
For arbitrary $J \in ([0,r]\cap \ZZ)^d$, there is a shortest element $g \in \Sigma_d$ such that $g^{-1} J$ is nondecreasing and set
\eq 
w^+_J =  w^+_{g^{-1} J} T_g,
\quad 
w^-_J = w^-_{g^{-1} J}  T_g
\endeq

\begin{lemma} \label{lem:Hyperaminoacidemia}
\begin{enumerate}[(a)]
\item For $I \in ([0, r]\cap \ZZ)^d$, $v_I u^+_{d}= w^+_I$.
\item For $I \in ([1, r]\cap \ZZ)^d$, $v_I u^-_{d}= w^-_I$.
\end{enumerate}
\end{lemma}

\begin{proof}
For non-decreasing $I$, the result follows from a direct computation. For general $I$, there exists a shortest element $g \in \Sigma_d$ such that $I g^{-1}$ is non-decreasing. Then, by Lemma~\ref{lem:Balladical}(a),  
\[
	v_I u^\pm_d = v_{I g^{-1}} T_g u^\pm_d = v_{I g^{-1}} u^\pm_d T_g = w^\pm_{I g^{-1}} T_g = w^\pm_I. \qedhere
\]
\end{proof}

\exa
Let $d = 7$ and let $I = (0,1,1,2,3,3,3)$. We have
\[
w^+_I = w^+_{0(0)} \otimes w^+_{1(0)} \otimes w^+_{1(1)} \otimes w^+_{2(0)} \otimes w^+_{3(0)} \otimes w^+_{3(1)} \otimes w^+_{3(2)}.
\]
For $J := (0,2,1,1,3,3,3) = I s_3 s_2$, 
\[
w^+_J =  w^+_I T_3 T_2.
\] 
\endexa 

\exa 
In the following we verify Lemma~\ref{lem:Hyperaminoacidemia} for small $d$'s.
%Let $d=1, I=(i)$ for some $i > 0$, we have 
%\[
%w^+_I = w^+_{1(0)} = v_{-i} + Qv_i,
%\quad
%w^-_I = w^-_{1(0)} = v_{-i} - Q\inv v_i.
%\]
%Since 
%$u^+_1 = T_0 + Q$ and $u^-_1 = T_0 - Q\inv$,
%indeed we have
%$
%v_I u^+_1 =   w^+_I$,
%and
%$v_I u^-_1 = v_{-i} - Q\inv v_i = w^-_I.$
Let $d=2, I = (1,1)$ and hence $w_I = w^+_{1(0)} \otimes w^+_{1(1)}$.
Since
$u^+_2 = (T_1 T_0 T_1 + Q) (T_0 + Q)$, we can check that indeed
\[
v_I u^+_2 = (v_1 \otimes v_1) (T_1 T_0 T_1 + Q) (T_0 + Q) 
=  (v_1 \otimes w^+_{1(1)}) (T_0 + Q) 
=  w^+_I.
\]
\endexa

Now we define $K$-vector spaces 
\eq 
W_{\geq 0}^d = V^{\otimes d} u^+_d, \quad W_{\lt 0}^d = V^{\otimes d} u^-_d. 
\endeq
By Lemma \ref{lem:Balladical}(a), $u^+_d$ and $u^-_d$ are in the center of $\cH_{Q, q}^\B(d)$, hence $W_{\geq 0}^d$ and $W_{\lt 0}^d$ are naturally $\cH_{Q, q}^\B(d)$-module via right multiplication. Moreover, $w T_0 = Q^{-1} w$ for all $w \in W_{\geq 0}^d$ and $w T_0 = - Q w$ for all $w \in W_{\lt 0}^d$.

\begin{lemma} \label{lem:Castoreum}
We have $W_{\geq 0}^d = V_{\geq 0}^{\otimes d} u_d^+$ and $W_{\lt 0}^d = V_{\gt 0}^{\otimes d} u_d^-$.
\end{lemma}

\begin{proof}
We only give a proof for the first claim, and a proof for the second claim can be obtained similarly. 
A direct computation shows that
\eq\label{eq:T0ud+}
T_0 u_d^+ = Q^{-1} u_d^+.
\endeq
For $1 \leq i \leq d$,
\begin{align*}
&(V_{\geq 0}^{\otimes (i - 1)} \otimes V_{\lt 0} \otimes V^{\otimes (d - i)}) u_d^+
& 
\\
&=(V_{\gt 0} \otimes V_{\geq 0}^{\otimes (i - 1)} \otimes V^{\otimes (d - i)}) T_0 T_1 T_2 \dots T_{i - 1} u_d^+
&
\\
&=  (V_{\gt 0} \otimes V_{\geq 0}^{\otimes (i - 1)} \otimes V^{\otimes (d - i)}) Q^{-1} T_1 T_2 \dots T_{i - 1} u_d^+ 
& \text{Lemma \ref{lem:Balladical}(a) and}\eqref{eq:T0ud+} 
\\
&\subseteq  V_{\geq 0}^{\otimes i} \otimes V^{\otimes (d - i)} u_d^+. 
&V_{\geq 0}^{\otimes i}\text{ is a $\cH_q(\Sigma_i)$-module}
\end{align*}
Next, an induction proves that for $0 \leq i \leq d$,
\eq
 V^{\otimes i} \otimes V^{\otimes (d - i)} = V_{\geq 0}^{\otimes i} \otimes V^{\otimes (d - i)}, 
\endeq
from which the result follows.
\end{proof}

\begin{lemma} \label{lem:Bountifulness}
Let $p_d : V^{\otimes d} \to V_{\leq 0}^{\otimes d}$ be the projection map. For $I \in ([0,r]\cap \ZZ)^d$ and $J \in ([1,r]\cap \ZZ)^d$, $p_d(w^+_I) = c_I v_{-I}$ and $p_d(w^-_J) = c_J v_{-J}$ for some invertible elements $c_I, c_J \in K^\times$.
\end{lemma}

\begin{proof}
When $I, J$ are non-decreasing, and when $d = 2$, the result follows from a direct computation. For general $I$ (or $J$), there exists a shortest element $g \in \Sigma_d$ such that $I g^{-1}$ (or $J g^{-1}$) is non-decreasing. The result follows from an induction on the length of $g$.
\end{proof}

\begin{lemma} \label{lem:Morphinomania}
\begin{enumerate}[(a)]
\item The map $v_I \mapsto w^+_I$ gives an isomorphism of $\cH_q(\Sigma_d)$-modules $V_{\geq 0}^{\otimes d} \to W_{\geq 0}^d$.
\item The map $v_I \mapsto w^-_I$ gives an isomorphism of $\cH_q(\Sigma_d)$-modules $V_{\lt 0}^{\otimes d} \to W_{\lt 0}^d$.
\end{enumerate}
\end{lemma}

\begin{proof}
Since $u^+_d$ (resp. $u^-_d$) is in the center of $\cH_{Q, q}^\B(d)$ by Lemma \ref{lem:Balladical}(a), the map $v_I \mapsto w^+_I$ (resp. $v_I \mapsto w^-_I$) is clearly $\cH_q(\Sigma_d)$-equivariant. Surjectivity of this map follows from Lemma \ref{lem:Castoreum}, and injectivity of this map follows from Lemma \ref{lem:Bountifulness}.
\end{proof}

\subsection{The actions of $v_{a, b}$}

\begin{lemma} \label{lem:Paddockstone}
For $a + b = d$, $V^{\otimes d} v_{a, b} = (V_{\gt 0}^{\otimes b} \otimes V_{\geq 0}^{\otimes a}) v_{a, b}$.
\end{lemma}

\begin{proof}
It follows from Eq. \eqref{eq:Reswarm} and Lemma \ref{lem:Castoreum} that
\eq 
V^{\otimes d} v_{a, b} = (V^{\otimes b} \otimes V^{\otimes a}) u_b^- T_{w_{a, b}} u_a^+ = (V_{\gt 0}^{\otimes b} \otimes V^{\otimes a}) u_b^- T_{w_{a, b}} u_a^+ = (V_{\gt 0}^{\otimes b} \otimes V^{\otimes a}) v_{a, b}. 
\endeq
For $b \lt i \leq d$,
\begin{align*}
& T_0 T_1 T_2 \dots T_{i - 1} v_{a, b} 
\\
&=  T_1^{-1} \dots T_b^{-1} (T_b \dots T_1 T_0 T_1 \dots T_b) (T_{b + 1} \dots T_{i - 1}) u_b^- T_{w_{a, b}} u_a^+ & \text{Eq. \eqref{eq:Reswarm}} 
\\
&=  T_1^{-1} \dots T_b^{-1} (T_b \dots  T_0  \dots T_b) u_b^- (T_{b + 1} \dots T_{i - 1}) T_{w_{a, b}} u_a^+ & \text{$u_b^-$ commutes with $T_{b + 1}, \dots$} 
\\
&=  T_1^{-1} \dots T_b^{-1} (u_{b + 1}^- + Q^{-1} u_b^-) (T_{b + 1} \dots T_{i - 1}) T_{w_{a, b}} u_a^+ & \text{Eq. \eqref{eq:Frigostable}} 
\\
&=  Q^{-1} T_1^{-1} \dots T_b^{-1} u_b^- (T_{b + 1} \dots T_{i - 1}) T_{w_{a, b}} u_a^+ & \text{Lemma \ref{lem:Balladical}} 
\\
&=  Q^{-1} T_1^{-1} \dots T_b^{-1} (T_{b + 1} \dots T_{i - 1}) u_b^- T_{w_{a, b}} u_a^+ & \text{$u_b^-$ commutes with $T_{b + 1}, \dots$} \\
&=  Q^{-1} T_1^{-1} \dots T_b^{-1} (T_{b + 1} \dots T_{i - 1}) v_{a, b}. & \text{Eq. \eqref{eq:Reswarm}}
\end{align*}
Then, for $b \lt i \leq d$,
\eq
\begin{split}
& (V_{\gt 0}^b \otimes V_{\geq 0}^{\otimes (i - b - 1)} \otimes V_{\lt 0} \otimes V^{\otimes (d - i)}) v_{a, b} 
\\
&= (V_{\gt 0} \otimes V_{\gt 0}^{\otimes b} \otimes V_{\geq 0}^{\otimes (i - b - 1)} \otimes V^{\otimes (d - i)}) T_0 T_1 T_2 \dots T_{i - 1} v_{a, b} 
\\
&= Q^{-1} (V_{\gt 0} \otimes V_{\gt 0}^{\otimes b} \otimes V_{\geq 0}^{\otimes (i - b - 1)} \otimes V^{\otimes (d - i)}) T_1^{-1} \dots T_b^{-1} (T_{b + 1} \dots T_{i - 1}) v_{a, b} 
\\
&\subseteq  (V_{\gt 0}^{\otimes b} \otimes V_{\gt 0} \otimes V_{\geq 0}^{\otimes (i - b - 1)} \otimes V^{\otimes (d - i)}) (T_{b + 1} \dots T_{i - 1}) v_{a, b} 
\\
&\subseteq  (V_{\gt 0}^{\otimes b} \otimes V_{\geq 0}^{\otimes (i - b)} \otimes V^{\otimes (d - i)}) v_{a, b},
\end{split}
\endeq
where the last two inclusions follow from that $V_{\gt 0}^{\otimes (b + 1)}$ is a $\cH_q(\Sigma_{b + 1})$-module, and $V_{\geq 0}^{\otimes (i - b)}$ is a $\cH_q(\Sigma_{i - b})$-module, respectively.
An induction shows that for $b \leq i \leq d$,
\[ V^{\otimes b} \otimes V^{\otimes (b - i)} \otimes V^{\otimes (d - i)} v_{a, b} = V_{\gt 0}^{\otimes b} \otimes V_{\geq 0}^{\otimes (i - b)} \otimes V^{\otimes (d - i)} v_{a, b}, \]
from which the result follows.
\end{proof}

For $a + b = d$, define projections
\eq 
p_{a, b} : V^{\otimes d} \to V_{\leq 0}^{\otimes a} \otimes V_{\lt 0}^{\otimes b},
\quad
p'_{a, b} : V^{\otimes d} \to V^{\otimes a} \otimes V_{\lt 0}^{\otimes b}.
 \endeq
\begin{lemma} \label{lem:Instar}
Let $a + b = d$, $I \in ([0, r] \cap \ZZ)^a$ and $J \in ([-r, r]\cap\ZZ)^b$.
Then
\[ p'_{a, b}( (v_J \otimes v_I) T_{w_{a, b}} ) = c_{I, J} v_I \otimes p_b(v_J) \]
for some invertible $c_{I, J} \in K^\times$, where $p_b$ is defined in Lemma \ref{lem:Bountifulness}. Moreover,
\[ p'_{a, b}( (w_J^- \otimes v_I) T_{w_{a, b}} ) = c_{I, J} c_J v_I \otimes v_{-J} \]
for some invertible $c_{I, J}, c_J \in K^\times$.
\end{lemma}
\begin{proof}
First note that $(v_J \otimes v_I) T_{w_{a, b}} = c_{I, J} (v_J \otimes v_I) w_{a, b} + \sum_{g < w_{a,b}} c_g (v_J \otimes v_I) g$ for some invertible $c_{I, J} \in K^\times$ and some $c_g \in K$, where $g < w_{a,b}$ under the Bruhat order. Hence, 
\eq\begin{split}
p'_{a, b}( (v_J \otimes v_I) T_{w_{a, b}} ) 
& =  p'_{a, b}( c_{I, J} (v_J \otimes v_I) w_{a, b} + \sum_{g \lt w_{a, b}} c_g (v_J \otimes v_I) g) 
\\
& =  c_{I, J} p'_{a, b}( v_I \otimes v_J ) + \sum_{g \lt w_{a, b}} c_g p'_{a, b}((v_J \otimes v_I) g)  \\
& =  c_{I, J} p'_{a, b}( v_I \otimes v_J ) = c_{I, J} v_I \otimes p_b(v_J).
\end{split}\endeq
By Lemma \ref{lem:Bountifulness}, $p_b(w_J^-) = c_J v_{-J}$ for some $c_J \in K^\times$. Therefore, $p'_{a, b}( (w_J^- \otimes v_I) T_{w_{a, b}} ) = c_{I, J} v_I \otimes p_b(w_J^-) = c_{I, J} c_J v_I \otimes v_{-J}$.
\end{proof}

\begin{lemma} \label{lem:Telephote}
For $I \in ([0, r] \cap \ZZ)^a$ and $J \in ([1, r] \cap \ZZ)^b$, $p_{a, b}((v_J \otimes v_I) v_{a, b}) = c v_{-I} \otimes v_{-J}$ for some $c \in K^\times$.
\end{lemma}

\begin{proof}
\begin{align*}
 p_{a, b}((v_J \otimes v_I) v_{a, b}) 
&= p_{a, b}((v_J \otimes v_I) u_b^- T_{w_{a, b}} u_a^+) & \text{Eq. \eqref{eq:Reswarm}} \\
&= p_{a, b}((w^-_J \otimes v_I) T_{w_{a, b}} u_a^+) & \text{Lemma \ref{lem:Hyperaminoacidemia}} \\
&= p_{a, b}(p'_{a, b}((w^-_J \otimes v_I) T_{w_{a, b}} ) u_a^+) \\
&= p_{a, b}(c_{I, J} c_J (v_I \otimes v_{-J}) u_a^+) & \text{Lemma \ref{lem:Instar}}\\
&= p_{a, b}(c_{I, J} c_J w^+_I \otimes v_{-J}) & \text{Lemma \ref{lem:Hyperaminoacidemia}} \\
&= c_{I, J} c_J p_a(w^+_I) \otimes v_{-J} \\
&= c_{I, J} c_I c_J v_{-I} \otimes v_{-J}. & \text{Lemma \ref{lem:Bountifulness}}
\end{align*}
\end{proof}

\begin{lemma} \label{lem:Chirivita}
For $a + b = d$, the map $v_I \otimes v_J \mapsto (v_J \otimes v_I) v_{a, b}$ gives an isomorphism of $\cH^\A(\Sigma_{a}) \otimes \cH^\A(\Sigma_{b})$-modules $V_{\geq 0}^a \otimes V_{\gt 0}^b \to V^{\otimes d} v_{a, b}$.
\end{lemma}

\begin{proof}
It follows from \cite[Lemma~3.10]{DJ89} that
\eq T_i v_{a, b} = \begin{cases}
		T_{i + a}, & 1 \leq i \leq b; \\
		T_{i - b}, & b + 1 \leq i \leq a + b - 1. \\
	\end{cases} \endeq
Hence the map is $\cH^\A(\Sigma_{a}) \otimes \cH^\A(\Sigma_{b})$-equivariant. The injectivity follows from Lemma \ref{lem:Telephote}, and the surjectivity follows from Lemma \ref{lem:Paddockstone}.
\end{proof}

\subsection{The proof}%of the isomorphism theorem}

Finally, we are in a position to prove the isomorphism theorem.
\begin{proof}[Proof of Theorem \ref{thm:SBA}]
\begin{align*}
S_{Q,q}^\B(n, d) =  {\rm End}_{\cH_{Q,q}^\B(d)}(V^{\otimes d}) 
= & {\rm End}_{\bigoplus_{0\leq i \leq d} e_{i, d - i} \cH_{Q,q}^\B(d) e_{i, d - i}}(V^{\otimes d} e_{i, d - i}) & \text{Lemma \ref{lem:Balladical}(e)}
\\
= & \bigoplus_{0\leq i \leq d} {\rm End}_{e_{i, d - i} \cH_{Q,q}^\B(d) e_{i, d - i}}(V^{\otimes d} e_{i, d - i})
\\
= & \bigoplus_{0\leq i \leq d} {\rm End}_{\cH_q^\A(\Sigma_i) \otimes \cH^\A_q(\Sigma_{d - i})}(V^{\otimes d} v_{i, d - i}) & \text{Lemma \ref{lem:Balladical}(c)(d)}
\\
= & \bigoplus_{0\leq i \leq d} {\rm End}_{\cH_q^\A(\Sigma_i) \otimes \cH^\A_q(\Sigma_{d - i})}(V_{\geq 0}^{\otimes i} \otimes V_{\gt 0}^{\otimes (d - i)}) & \text{Lemma \ref{lem:Chirivita}} \\
= & \bigoplus_{0\leq i \leq d} {\rm End}_{\cH^\A_q(\Sigma_i)}(V_{\geq 0}^{\otimes i}) \otimes {\rm End}_{\cH^\A_q(\Sigma_{d - i})}(V_{\gt 0}^{\otimes d - i}) \\
= & \bigoplus_{0\leq i \leq d} S_q^\A(\lceil n / 2 \rceil, i) \otimes S_q^\A(\lfloor n / 2 \rfloor, d - i).  & \text{Eq. \eqref{eq:Gyroscopic}}
\end{align*}
\end{proof}

\subsection{Simple modules of $S_{Q,q}^\B(n,d)$} As an immediate consequence of the isomorphism theorem one obtains a classification of irreducible representations for $S_{Q,q}^\B(n,d)$. 

\begin{thm}\label{thm:ModSB} If $f^\B_d(Q, q)$ is invertible in the field $K$ then there is a bijection
\[
\{\textup{Irreducible representations of }S_{Q,q}^\B(n,d)\} \leftrightarrow 
\{(\lambda,\mu) \vdash (d_1, d_2) ~|~ d_1 + d_2 = d\},
\]
where number of parts of $\lambda$ and $\mu$ is no more than $n$. In particular, the standard modules over $S_{Q,q}^\B(n,d)$ are of the form $\nabla(\lambda)\boxtimes \nabla(\mu)$, where $\nabla(\lambda)$ (resp. $\nabla(\lambda)$) are standard modules over $S_q^\A(\nhc,d_1)$ (resp. $S_q^\A(\nhf,d_2)$).
\end{thm}
\rmk
There are variants of our isomorphism theorem in the literature related to different Schur algebras. 
In \cite{GH97} it is established a Morita equivalence
\eq
S(\cH(W^\B_d))
\Mor
\bigoplus_{i=0}^d S_q^\A(i, i) \otimes S_q^\A(d-i, d - i),
\endeq
where $S(\cH(W^\B_d))$ is an endomorphism algebra on a $q$-permutation module involving $r$-compatible compositions.

By \cite[Theorem~3.2]{A99}, under a separation condition at the specialization $u_1 = -Q, u_2 = Q^{-1}$, 
the Sakamoto-Shoji algebra $S^\B_{u_1, u_2, q}(\nhc, \nhf, d)$ (see \ref{sec:SS}) is isomorphic to the right hand side of Theorem~\ref{thm:SBA}, and hence is isomorphic to our algebra $S^\B_{Q,q}(n,d)$.
When the separation condition fails, the two algebras do not coincide since their dimensions do not match. 
For instance, in \cite[\S2, Example]{A99} it is computed that $\dim S^\B_{u_1,u_2,q}(1,1,2)$ can be 3, 4, 5 and 10 while $\dim S^\B_q(2,d)$ is always 3.
\endrmk

%=============================================
\section{Schur Functors} \label{sec:SF}
%=============================================

%=============================================
\subsection{Schur functors}
%=============================================
For type A it is well-known that, provided $n \geq d$, there is an idempotent $e^\A = e^\A(n,d) \in S^\A_q(n,d)$
 such that $e^\A S^\A_q(n,d) e^\A \simeq \cH_q(\Sigma_{d})$, and a Schur functor
\eq\label{def:SFA}
F^\A_{n,d}: \Mod(S^\A_{q}(n,{d})) \to \Mod(\cH_{q}(\Sigma_{d})),
\quad
M\mapsto e^\A M.
\endeq

In the following proposition we construct the Schur functor for $S^\B_{Q,q}(n,d)$ when $\nhf \geq d$.

%-----------------------------------------------------------------------------------------------
\prop\label{prop:SFB}
If $\nhf \geq d$ then there is an idempotent  $e^\B = e^\B(n,d) \in  S^\B_{Q,q}(n,d)$ such that $e^\B S^\B_{Q,q}(n,d) e^\B \simeq \cH^\B_{Q,q}(d)$ as $K$-algebras, and $e^\B S^\B_{Q,q}(n,d) \simeq V^{\otimes d}$ as $(S^\B_{Q,q}(n,d), \cH^\B_{Q,q}(d))$-bimodules.. 
\endprop
%-----------------------------------------------------------------------------------------------
\proof
Recall $\Lambda^\B(n,d)$ from \eqref{def:LBnd} and $\phi^g_{\lambda\mu}$ from \eqref{def:phiB}.
Let
$ e^\B = \phi^1_{\omega\omega}$,
where 
\eq
\omega =\Bc{
\{
(0, \ldots, 0, \underset{2d}{\underbrace{1, \ldots, 1}}, 0 \ldots, 0) \in \Lambda^\B(2r,d)
\}
&\tif n = 2r;
\\
\{
 (0, \ldots, 0, \underset{2d+1}{\underbrace{1, \ldots, 1}}, 0 \ldots, 0) \in \Lambda^\B(2r+1,d)
\}
&\tif n = 2r+1.
}
\endeq
Note that such $\omega$ is well-defined only when $r = \nhf \geq d$.
By Lemma~\ref{lem:SBmult}, we have 
\eq
e^\B \phi_{\lambda\mu}^g e^\B = \Bc{
\phi_{\lambda\mu}^g&\tif \lambda = \omega = \mu;
\\
0&\otw.
}
\endeq
Since $W_\omega$ is the trivial group, $x_\omega = 1 \in \cH^\B_{Q,q}(d)$ and hence
$
\phi^g_{\omega\omega}$ is uniquely determined by $1 \mapsto T_g$. Therefore,
 $e^\B S^\B_{Q,q}(n, d)e^\B $ and $\cH^\B_{Q,q}(d)$ are isomorphic as algebras.
 
Now from Section~\ref{sec:comb} we see that there is a canonical identification 
\eq
V^{\otimes d}
\simeq 
\bigoplus_{\mu \in \Lambda^\B(n,d)} x_\mu \cH^\B_{Q,q} 
\simeq
\bigoplus_{\mu \in \Lambda^\B(n,d)} \Hom_{\cH^\B_{Q,q}}(x_\omega \cH^\B_{Q,q}, x_\mu \cH^\B_{Q,q}),
\endeq
and hence the maps $\phi^g_{\omega\mu}$, with $\mu \in \Lambda^\B(n,d)$, $g$ a minimal length coset representative for $W^\B/W_\mu$, forms a linear basis for $V^{\otimes d}$.
Again by Lemma~\ref{lem:SBmult}, we have 
\eq\label{eq:eBphi}
e^\B \phi_{\lambda\mu}^g  = \Bc{
\phi_{\omega\mu}^g&\tif \lambda = \omega;
\\
0&\otw.
}
\endeq
Hence, $e^\B S^\B_{Q,q}(n,d)$ has a linear basis $\{\phi^g_{\omega\mu}\}$ where $\mu \in \Lambda^\B(n,d)$, $g$ a minimal length double coset representative for $W_\omega \backslash W^\B/W_\mu$. Therefore $V^{\otimes d}$ and $e^\B S^\B_{Q,q}(n,d)$ are isomorphic as $(S^\B_{Q,q}(n,d), \cH^\B_{Q,q}(d))$-bimodules.
\endproof
%-----------------------------------------------------------------------------------------------
We define the Schur functor of type B by
\eq
F^\B_{n,d}: \Mod(S^\B_{Q,q}(n,d)) \to \Mod(\cH^\B_{Q,q}(d)),
\quad
M\mapsto e^\B M.
\endeq
Define the inverse Schur functor by
\eq
G^\B_d: \Mod(\cH^\B_{Q,q}(d)) \to \Mod(S^\B_{Q,q}(n,d)),
\quad
M 
\mapsto
\Hom_{e^\B S^\B_{Q,q}(n,d) e^\B}(e^\B S^\B_{Q,q}(n,d), M).
\endeq

In below we define a Schur-like functor $F^\flat_{n,d}: \Mod(S^\B_{Q,q}(n,d)) \to \Mod(\cH^\B_{Q,q}(d))$ using Theorem~\ref{thm:SBA}, under the same invertibility assumption: recall $\Phi$ from \eqref{eq:Phi},
let 
\eq
\epsilon^\flat = \epsilon^\flat_{n,d} = \Phi^{-1}(\bigoplus_{i=0}^d e^\A(\nhc,i) \otimes e^\A(\nhf, d-i)).
\endeq
Note that $\epsilon^\flat S^\B_{Q,q}(n,d)\epsilon^\flat \simeq \bigoplus_{i=0}^d \cH_q(\Sigma_{i+1}) \otimes \cH_q(\Sigma_{d-i+1})$, and hence left multiplication by $\epsilon^\flat$ defines a functor
$\Mod(S^\B_{Q,q}(n,d)) \to \Mod(\bigoplus_{i=0}^d \cH_q(\Sigma_{i+1}) \otimes \cH_q(\Sigma_{d-i+1}))$.
Hence, we can define
\eq\label{def:Fflat}
F^\flat_{n,d}: \Mod(S^\B_{Q,q}(n,d)) \to \Mod(\cH^\B_{Q,q}(d)),
\quad
M \mapsto \cF_H^{-1}(\epsilon^\flat M),
\endeq
where $\cF_H$ is the Morita equivalence for the Hecke algebras given by
\eq\label{def:FH}
\cF_H: \textup{Mod}(\cH^\B_{Q,q}(d)) \to \textup{Mod}
\Big(
\bigoplus_i \cH_{q}(\Sigma_{i+1})\otimes \cH_{q}(\Sigma_{d-i+1})
\Big).
\endeq
Under the invertibility condition, one can define an equivalence of categories induced from $\Phi$ as below:
\eq\label{def:FS}
\cF_S: \textup{Mod}(S^\B_{Q,q}(n,d)) \to \textup{Mod}
\Big(
\bigoplus\limits_{i=0}^d S^\A_q(\nhc, i)\otimes S^\A_q(\nhf, d-i)
\Big).
\endeq
In other words, we have the following commutativity of functors:
%-----------------------------------------------------------------------------------------------
\begin{prop}\label{prop:SF}
Assume $\nhf \geq d \geq i \geq 0$ and that $f^\B_q$ is invertible. The diagram below commutes:
\eq\label{eq:SF}
\begin{tikzcd}
\textup{Mod}(S^\B_{Q,q}(n,d)) \arrow{r}{\cF_s} \arrow{d}{F^\flat_{n,d}}& \textup{Mod}
\Big(
\bigoplus\limits_{i=0}^d S^\A_q(\nhc, i)\otimes S^\A_q(\nhf, d-i)
\Big) \arrow{d}{\bigoplus\limits_{i=0}^d F^\A_{\nhc,i} \otimes F^\A_{\nhf, d-i}}
\\
\textup{Mod}(\cH^\B_{Q,q}(d)) \arrow{r}{\cF_H} & \textup{Mod}
\Big(
\bigoplus\limits_{i=0}^d \cH_{q}(\Sigma_{i+1})\otimes \cH_{q}(\Sigma_{d-i+1})
\Big)
\end{tikzcd}
\endeq
\end{prop}
%-----------------------------------------------------------------------------------------------
\rmk
We expect that Proposition~\ref{prop:SF} still holds if we replace the functor $F^\flat_{n,d}$ therein by $F^\B_{n,d}$. 
\endrmk
%-----------------------------------------------------------------------------------------------
\subsection{Existence of idempotents}
%-----------------------------------------------------------------------------------------------
We construct additional idempotents in Schur algebras of type B that will be used later in Section~\ref{sec:Rep}.

%-----------------------------------------------------------------------------------------------
\prop\label{prop:idem}
There exists an idempotent  $e  \in  S^\B_{Q,q}(n',d)$ such that $e S^\B_{Q,q}(n',d) e \simeq S^\B_{Q,q}(n,d)$ if either one of the following holds:
\enu[(a)]
\item $n' \geq n$ and $n' \equiv n \mod 2$;
\item $n' = 2r' +1 \geq n = 2r$.
\endenu
\endprop
%-----------------------------------------------------------------------------------------------
\proof
We use the combinatorial realization in Section~\ref{sec:comb}. 
For (a)
we set
\eq 
e =  \sum_{\gamma}\phi^1_{\gamma\gamma},
\endeq 
where $\gamma$ runs over the set
\eq
\Lambda^\B(n',d)|_{n} = 
\Bc{
\{
\gamma = (0, \ldots, 0, \underset{n}{\underbrace{*, \ldots, *}}, 0 \ldots, 0) \in \Lambda^\B(n',d)
\}
&\tif (a) \textup{ holds};
\\
\{
\gamma = (0, \ldots, 0, \underset{r}{\underbrace{*, \ldots, *}}, 1, \underset{r}{\underbrace{*, \ldots, *}}, 0 \ldots, 0) \in \Lambda^\B(n',d)
\}
&\tif (b) \textup{ holds}.
}
\endeq

By Lemma~\ref{lem:SBmult} we have 
\eq
e \phi_{\lambda\mu}^g e = \Bc{
\phi_{\lambda\mu}^g&\tif \lambda, \mu \in \Lambda^\B(n',d)|_{n};
\\
0&\otw.
}
\endeq
It follows by construction that $eS^\B_{Q,q}(n', d)e$ and $S^\B_{Q,q}(n, d)$ are isomorphic as algebras.
\endproof
%-----------------------------------------------------------------------------------------------

\subsection{Existence of spectral sequences}
%-----------------------------------------------------------------------------------------------

Let $A$ be a finite dimensional algebra over a field $k$ and $e$ be an idempotent in $A$. 
Doty, Erdmann and Nakano \cite{DEN04} established a relationship between 
the cohomology theory in $\text{Mod}(A)$ versus $\text{Mod}(eAe)$. More specifically, 
they construct a Grothendieck spectral sequence which starts from extensions of $A$-modules and converges
to extensions of $eAe$-modules. 

There are two important functors involved in this construction. The first functor is an exact functor from $\text{Mod}(A)$ to 
$\text{Mod}(eAe)$ denoted by ${\mathcal F}$ (that is a special case of the classical Schur functor) defined by ${\mathcal F}(-)=e(-)$.   
The other functor is a left exact functor from $\text{Mod}(eAe)$ to $\text{Mod}(A)$, denoted ${\mathcal G}$ defined by 
${\mathcal G}(-)=\text{Hom}_{A}(Ae,-)$. This functor is right adjoint to ${\mathcal F}$. 

In \cite{DEN04}, the aforementioned construction was used in the quantum setting to 
relate the extensions for quantum $\GL_{n}$ to those for Hecke algebras. For $\nhf \geq d$ there exists an 
idempotent $e\in S_{Q,q}^{\B}(n,d)$ such that $\cH_{Q,q}^\B(d)\cong eS_{Q,q}^{\B}(n,d)e$. Therefore, we obtain 
a relationship between cohomology of the type B Schur algebras with the Hecke algebra of type B. 

\begin{thm} Let $\nhf \geq d$ with $M\in \operatorname{Mod}(S_{Q,q}^{\B}(n,d))$ and $N\in \operatorname{Mod}(\cH_{Q,q}^\B(d))$. 
There exists a first quadrant spectral sequence 
$$E_{2}^{i,j}=\operatorname{Ext}^{i}_{S_{Q,q}^{\B}(n,d)}(M,R^{j}{\mathcal G}(N))\Rightarrow \operatorname{Ext}^{i+j}_{\cH_{Q,q}^\B(d)}(eM,N).$$ 
where $R^{j}{\mathcal G}(-)=\operatorname{Ext}_{\cH_{Q,q}^\B(d) }^{j}(V^{\otimes d},N)$. 
\end{thm} 

We can also compare cohomology between $S_{Q,q}^{\B}(n,d)$ and $S_{Q,q}^{\B}(n^{\prime},d)$ where $n^{\prime}\geq n$ since there exists 
an idempotent $e\in S_{Q,q}^{\B}(n^{\prime},d)$ such that $S_{Q,q}^{\B}(n,d)\cong eS_{Q,q}^{\B}(n^{\prime},d)e$ thanks to Proposition~\ref{prop:idem}.  

\begin{thm} Let $M\in \operatorname{Mod}(S_{Q,q}^{\B}(n^{\prime},d))$ and $N\in \operatorname{Mod}(S_{Q,q}^{\B}(n,d))$. 
Assume that either 
\enu[(a)]
\item $n' \geq n$ and $n' \equiv n \mod 2$;
\item $n' = 2r' +1 \geq n = 2r$.
\endenu
Then there exists a first quadrant spectral sequence 
$$E_{2}^{i,j}=\operatorname{Ext}^{i}_{S_{Q,q}^{\B}(n^{\prime},d)}(M,R^{j}{\mathcal G}(N))\Rightarrow \operatorname{Ext}^{i+j}_{S_{Q,q}^{\B}(n,d)}(eM,N).$$ 
where $R^{j}{\mathcal G}(-)=\operatorname{Ext}_{S_{Q,q}^{\B}(n,d)}^{j}(eS_{Q,q}^{\B}(n^{\prime},d),-)$. 
\end{thm}

%=============================================
\section{Cellularity}\label{sec:cell}
%=============================================
%-----------------------------------------------------------------------------------------------
\subsection{Definition}
%-----------------------------------------------------------------------------------------------
We start from recalling the definition of a cellular algebra following \cite{GL96}.  
A $K$-algebra $A$ is {\em cellular} if it is equipped with a cell datum $(\Lambda, M, C, *)$ consisting of a poset $\Lambda$, 
a map $M$ sending each $\lambda \in \Lambda$ to a finite set $M(\lambda)$, 
a map $C$ sending each pair $(\fs,\ft) \in M(\lambda)^2$ to an element $C_{\fs,\ft}^\lambda \in A$, 
and an $K$-linear involutory anti-automorphism $*$ satisfying the following conditions:
\enu
\item[(C1)]
The map $C$
%\[
%C : \coprod_{\lambda \in \Lambda} \left( M(\lambda) \times M(\lambda) \right) \rightarrow A,
%\quad
%(\fs,\ft) \mapsto C^\lambda_{\fs,\ft} 
%\quad
%\] 
is injective with image being an $K$-basis of $A$ (called a {\em cellular basis}).
\item[(C2)]
%The map $*$ 
%is an $K$-linear involutory anti-automorphism 
%of $A$ such that
For any $\lambda \in \Lambda$ and $\fs, \ft \in M(\lambda)$,
$(C_{\fs, \ft}^{\lambda})^* = C_{\ft, \fs}^{\lambda}$.
\item[(C3)]
There exists
$r_a (\fs', \fs) \in K$ for $\lambda\in \Lambda, \fs, \fs' \in M(\lambda)$
such that for all $a \in A$ and $\fs, \ft \in M(\lambda)$,
\[
a  C_{\fs, \ft}^{\lambda} \equiv \sum_{\fs' \in M(\lambda)} r_a (\fs', \fs) C_{\fs', \ft}^{\lambda}
\mod A_{< \lambda}.
\] 
Here $A_{< \lambda}$ is the
$K$-submodule of $A$ generated by the set 
$
\{ C_{\fs'', \ft''}^{\mu} \mid \mu < \lambda; \fs'', \ft'' \in M(\mu) \}.
$
\endenu
For a cellular algebra $A$, we define for each $\lambda \in \Lambda$ a {\em cell module} $W(\lambda)$ spanned by $C_{\fs}^{\lambda}$, $\fs \in M(\lambda)$, with multiplication given by
\eq\label{eq:cellmod}
a C_\fs = \sum_{\fs'\in M(\lambda)} r_a(\fs',\fs) C_\fs'.
\endeq
For each $\lambda \in \Lambda$ we let 
$\phi_\lambda: W(\lambda) \times W(\lambda) \to K$ 
be a bilinear form satisfying
\eq\label{eq:phi_lambda}
C^\lambda_{\fs,\fs} C^\lambda_{\ft,\ft} \equiv \phi_\lambda(C_\fs, C_\ft) C^\lambda_{\fs,\ft} 
\mod A_{<\lambda}.
\endeq
%The set below parametrizes the irreducible modules of $A$:
%\eq
%\Lambda_0 = \{\lambda \in \Lambda ~|~ \phi_\lambda \neq 0\}.
%\endeq
%=============================================
%\subsection{Cellular structures on $\cH_q^\B(d)$}
%=============================================

It is known that the type A $q$-Schur algebras are always cellular, and there could be distinct cellular structures. 
See \cite{AST18} for a parallel approach on the cellularity of centralizer algebras for quantum groups.

\exa[Mathas]\label{C1}
Let $\Lambda = \Lambda^\A(d)$ be the set of all partitions of $d$, and let $\Lambda' = \Lambda'(d)$ be the set of all compositions of $d$. For each composition $\lambda \in \Lambda'$, let $\Sigma_\lambda$ be the corresponding Young subgroup of $\Sigma_d$. We set
\[
x_\lambda = \sum_{w \in \Sigma_\lambda} T_w \in \cH_q(\Sigma_d).
\]
It is known the $q$-Schur algebra admits the following combinatorial realization:
\[
S_q^\A(n,d) = \textup{End}_{\cH_q(\Sigma_d)}(\oplus_{\lambda\in\Lambda'}x_\lambda\cH_q(\Sigma_d))
=
\bigoplus_{\lambda,\mu\in\Lambda'} \textup{Hom}_{\cH_q(\Sigma_d)}(x_\mu\cH_q(\Sigma_d),x_\lambda\cH_q(\Sigma_d)).
\]
%\eq\label{eq:C1L}
%\Lambda = \{\lambda = (\lambda_1 \geq \lambda_2 \geq \ldots ) ~|~ \sum_{i \geq 1}\lambda_i=d\}.
%\endeq
The finite set $M(\lambda)$ is given by $M(\lambda)  = \bigsqcup_{\mu\in\Lambda'} \textup{SSTD}(\lambda,\mu)$, where
\eq\label{eq:C1M}
\textup{SSTD}(\lambda,\mu)  = \{\textup{semi-standard }\lambda\textup{-tableaux of shape }\mu\}.
\endeq
For $\mu \vdash d$, denote the set of shortest right coset representatives for $\Sigma_{\mu}$ in $\Sigma_{d}$ by
\eq
D_\mu =\{ w\in \Sigma_d ~|~ \ell(gw) = \ell(w) + \ell(g) \textup{~for all~} g\in \Sigma_\mu\}.
\endeq
Let $\ft^\lambda$ be the canonical $\lambda$-tableau of shape $\lambda$, then for all $\lambda$-tableau $\ft$ there is a unique element $d(\ft) \in D_\lambda$ such that $\ft d(\ft) = \ft$.
The cellular basis element, for $\lambda \in \Lambda, \fs \in \textup{sstd}(\lambda, \mu), \ft \in \textup{sstd}(\lambda,\nu)$, is the given by
\eq\label{eq:C1C}
C_{\fs,\ft}^\lambda(x_\alpha h) = \delta_{\alpha, \mu} \sum_{s,t}T_{d(s)^{-1}} x_\lambda T_{d(t)} h,
\endeq
where the sum is over all pairs $(s,t)$ such that $\mu(s) = \fs, \nu(t) = \ft$.
\endexa
\exa[Doty-Giaquinto]
The poset $\Lambda$ is the same as in Example~\ref{C1}, and we have $\Lambda = \Sigma_d \Lambda^+$.
It is known that the algebra $S_q^\A(n,d)$ admits a presentation with generators $E_i, F_i (1\leq i\leq n-1)$ and $1_\lambda (\lambda \in \Lambda)$. The map $*$ is the anti-automorphism satisfying
\[
E_i^* = F_i,
\quad
F_i^* = E_i,
\quad
1_\lambda^* = 1_\lambda.
\]
For each $\lambda \in \Lambda$ we set
$\Lambda^+_\lambda = \{\mu \in \Lambda^+ ~|~ \mu \leq \lambda\}$.
Note that $\Lambda^+_\lambda$ is saturated and it defines a subalgebra $S_q(\Lambda^+_\lambda)$ of $S_q^\A(n,d)$ with a basis $\{\={x}_\fs~|~ 1 \leq \fs \leq d_\lambda\}$ for some $d_\lambda \in \NN$.
Let $x_\fs \in S_q^\A(n,d)^-$ be the preimage of $\={x}_\fs$ under the projection $S_q^\A(n,d) \to S_q(\Lambda^+_\lambda)$ that is the identity map except for that it kills all $1_\mu$ where $\mu \not\leq \lambda$.
The finite set $M(\lambda)$ is given by
\eq\label{eq:C2M}
M(\lambda) = \{1, 2, \ldots, d_\lambda\}.
\endeq
Finally, for $\lambda \in \Lambda, \fs, \ft \in M(\lambda)$, we set
\eq\label{eq:C2C}
C_{\fs,\ft}^\lambda = 
x_\fs 1_\lambda x_\ft^*.
\endeq

\endexa
\subsection{Cellular structures on $S_{Q,q}^\B(n,d)$}
%=============================================
We show that the isomorphism theorem produces a cellular structure for $S_{Q,q}^\B(n,d)$ using any cellular structure on the $q$-Schur algebras of type A.
For any $n,d$ we fix a cell datum $(\Lambda_{n,d}, M_{n,d}, C_{n,d}, *)$ for $S_q^\A(n,d)$. 
Define 
\eq\label{def:LB}
\Lambda^\B = \Lambda^\B(n,d) = 
\bigsqcup_{i=0}^d 
\Lambda_{\nhc ,i} \times \Lambda_{\nhf ,d-i},
\endeq
as a poset with the lexicographical order.
For $\lambda = (\lone, \ltwo) \in \Lambda^\B$, we define $M^\B$ by
\eq
M^\B(\lambda) = \bigsqcup_{i=0}^d M_{\nhc,i}(\lone) \times M_{\nhf,d-i}(\ltwo).
\endeq
The map $C^\B$ is given by, for $\fs = (\fsone, \fstwo), \ft = (\ftone,\fttwo) \in M_{\nhc,i}(\lone) \times M_{\nhf,d-i}(\ltwo) \subset M^\B(\lambda)$,
\eq
(C^\B)^\lambda_{\fs,\ft}=
(C_{\nhc,i})^{\lone}_{\fsone,\ftone} \otimes (C_{\nhf,d-i})^{\ltwo}_{{\fstwo, \fttwo}}.
\endeq
Finally, the map $*$ is given by
\eq\label{eq:B*}
*:
(C_{\nhc,i})^{\lone}_{\fsone,\ftone} \otimes (C_{\nhf,d-i})^{\ltwo}_{\fstwo, \fttwo}
\mapsto
(C_{\nhc,i})^{\lone}_{\ftone,\fsone} \otimes (C_{\nhf,d-i})^{\ltwo}_{\fttwo, \fstwo}.
\endeq
%-----------------------------------------------------------------------------------------------
\cor
If the invertibility condition in Theorem~\ref{thm:SBA} holds, then
$S_{Q,q}^\B(n,d)$ is a cellular algebra with cell datum $(\Lambda^\B, M^\B, C^\B, *)$.
\endcor
%-----------------------------------------------------------------------------------------------
\proof
Condition (C1) follows from the isomorphism theorem; while
Condition (C2) follows directly from \eqref{eq:B*}. 
Condition (C3) follows from the  type A  cellular structure as follows: for 
$a_1 \in S_q^\A(\nhc,i)$ and $a_2 \in S_q^\A(\nhf,d-i)$,  
\begin{align*}
a_1  (C_{\nhc,i})_{\fsone, \ftone}^{\lone} 
&\equiv \sum_{\fuone \in M_{\nhc,i}(\lone)} 
r^{(1)}_{a_1} (\fuone, \fsone) 
(C_{\nhc,i})_{\fuone, \ftone}^{\lone}
\mod S_q^\A(\nhc,i)(< \lone),
\\
a_2  (C_{\nhf,d-i})_{\fstwo, \fttwo}^{\ltwo} 
&\equiv \sum_{\futwo \in M_{r,d-i}(\ltwo)} 
r^{(2)}_{a_2} (\futwo, \fstwo) 
(C_{\nhf,d-i})_{\futwo, \fttwo}^{\ltwo}
\mod S_q^\A(\nhf,d-i)(< \ltwo).
\end{align*}
That is, for $a = a_1 \otimes a_2 \in S_q^\A(\nhc,i) \otimes S_q^\A(\nhf,d-i) \subset S_q^\B(n,d)$ we have
\[
a(C^\B)^\lambda_{\fs,\ft}
\equiv \sum_{\substack{\fu= (\fuone,\futwo) \\ \in M_{\nhc,i}(\lone)\times M_{\nhf,d-i}(\ltwo)}} 
r^\B_a(\fu, \fs)
(C^\B)^\lambda_{\fu,\ft}
%(C_{r,i})_{\fuone, \ftone}^{\lone}
%\otimes 
% (C_{r,d-i})_{\futwo, \fttwo}^{\ltwo}
\mod S_q^\B(n,d)(< \lambda),
\]
where $r^\B_a(\fu, \fs) = r^{(1)}_{a_1} (\fuone, \fsone) r^{(2)}_{a_2} (\futwo, \fstwo)$ is independent of $\ft$.
\endproof
\section{Quasi-Hereditary Structure}
%=============================================
%=============================================
\subsection{Definition}
%=============================================
%For a finite-dimensional $K$-algebra $A$, let $\{L(\lambda)~|~ \lambda\in\Lambda_0\}$ be the full set of absolutely irreducible modules. 
%Let $P(\lambda)$ be the projective cover of $L(\lambda)$, and let $\Delta(\lambda)$ be the largest quotient of $P(\lambda)$ with composition factors $L(\mu)$ such that $\mu \leq \lambda$. 
%The first definition for $A$ to be quasi-hereditary is that, for each $\lambda \in \Lambda_0,$
%\enu
%\item[(Q1)] $[\Delta(\lambda) : L(\lambda)] = 1$;
%\item[(Q2)] $P(\lambda)$ has a $\Delta$-filtration such that
%\[
%(P(\lambda): \Delta(\lambda)) =1,
%\quad
%(P(\lambda): \Delta(\mu)) \neq 0  \Rightarrow \mu \geq \lambda.
%\]
%\endenu
Following \cite{CPS88},
a $K$-algebra $A$ is called {quasi-hereditary} if there is a chain of two-sided ideals of $A$:
\[
0 \subset I_1 \subset I_2 \subset \ldots \subset I_n = A
\]
such that each quotient $J_j = I_j/I_{j-1}$ is a hereditary ideal of $A/I_{j-1}$.
%We remark that an ideal $J$ of $A$ is {\em hereditary} if 
%\enu
%\item[(H1)] $J^2 = J$; 
%\item[(H2)] $J \textup{rad}(A) J = 0$;
%\item[(H3)] $J$ is a projective left (or, right) $A$-module.
%\endenu
%\rmk\label{rmk:CtoQ}
%Here are some facts that might be useful:
It is known \cite{GL96} that if $A$ is cellular and $\phi_\lambda \neq 0$ (cf. \eqref{eq:phi_lambda}) for all $\lambda \in \Lambda$ then $A$ is quasi-hereditary.
%\enu
%\item If $A$ is cellular and $\phi_\lambda \neq 0$ (cf. \eqref{eq:phi_lambda}) for all $\lambda \in \Lambda$ then $A$ is quasi-hereditary;
%
%\item If $J$ is a cell ideal of $A$ then $J$ satisfies exactly one condition below:
%\enu
%\item $J^2 = 0$;
%\item There is a primitive idempotent $e\in A$ that generates $J$ as a two-sided ideal.
%\endenu
%Moreover, (b) implies that $J^2 = J \simeq Ae \otimes_K eA, eAe \simeq K$, and $J$ is a hereditary ideal of $A$.
%\item Let $A$ be cellular. Then TFAE:
%\enu
%\item $A$ has finite global dimension;
%\item The Cartan matrix of $A$ has determinant one;
%\item $A$ is quasi-hereditary;
%\item $\Ext_A^2(W(\lambda), W(\mu)^*) = 0$ for all $\lambda,\mu \in \Lambda$;
%\endenu
%\endenu
%\endrmk

An immediate corollary of our isomorphism theorem is that $S_{Q,q}^\B(n,d)$ is quasi-hereditary under the invertibility condition.  
We conjecture that this is a sufficient and necessary condition and provide some evidence for small $n$.

%-----------------------------------------------------------------------------------------------
%-----------------------------------------------------------------------------------------------
\begin{cor}
If the invertibility condition in Theorem~\ref{thm:SBA} holds, then
$S_q^\B(n,d)$ is quasi-hereditary.
\end{cor}
%-----------------------------------------------------------------------------------------------
\proof
%We use the fact that for any cellular algebra $A$, if $\phi_\lambda \neq 0$ (see \eqref{eq:phi_lambda}) for all $\lambda \in \Lambda$ then $A$ is quasi-hereditary.
Let $\phi^\A_\nu$ with $\nu \in \Lambda_{r,j}$ be such a map for $S_q^\A(r,j)$.
Fix $\lambda = (\lone, \ltwo) \in \Lambda_{\nhc,i}\times \Lambda_{\nhf,d-i}\subset \Lambda^\B$,
$\fs =(\fsone, \fstwo), \ft = (\ftone,\fttwo) \in M_{\nhc,i}(\lone)\times M_{\nhf,d-i}(\ltwo) \subset M^\B(\lambda)$, we have
\begin{align*}
C^\lambda_{\fs,\fs} C^\lambda_{\ft,\ft} 
&=
(C_{\nhc,i})^{\lone}_{\fsone,\fsone} (C_{\nhc,i})^{\lone}_{\ftone,\ftone}  \otimes (C_{\nhf,d-i})^{\ltwo}_{{\fstwo, \fstwo}} (C_{\nhf,d-i})^{\ltwo}_{{\fttwo, \fttwo}}
\\
&\equiv \phi^\A_{\lone}(C_\fsone, C_\ftone)\phi^\A_{\ltwo}(C_\fstwo, C_\fttwo) C^\lambda_{\fs,\ft} 
\mod S_q^\B(n,d)(<\lambda).
\end{align*}
\endproof
%-----------------------------------------------------------------------------------------------
Recall that in Proposition~\ref{prop:SB(2,1)} we see that $S_{Q,q}^\B(2,1) \simeq \cH^\A_{Q^{-1}}(\Sigma_2)$. 
In the following we show that the known cellular structure (due to Geck/Dipper-James) fails when $f^\B = Q^{-2} + 1$ is not invertible. 
%-----------------------------------------------------------------------------------------------
\exa
Let $S_{Q,q}^\B(2,1) \simeq \cH_{Q^{-1}}^\A(\Sigma_2) = K[t]/\langle t^2 - (Q\inv - Q)t + 1\rangle$. We have
\[
\Lambda = \left\{\lambda = \iyng{2} \rhd \mu = \iyng{1,1}\right\},
\quad
M(\lambda) = \{\ft = \iyoung{12}\},
M(\mu) = \left\{\fs = \iyoung{1,2}\right\}.
\]
The cellular basis elements are
\[
C^\lambda_{\ft\ft} = \sum_{w \in \Sigma_2} Q^{-\ell(w)} T_w = 1 + Q\inv t,
\quad
C^\mu_{\fs\fs} = \sum_{w \in \Sigma_1 \times \Sigma_1} Q^{-\ell(w)} T_w = 1.
\]
Firstly, we have
$C^\mu_{\fs\fs}C^\mu_{\fs\fs} = 1 = C^\mu_{\fs\fs}$ and hence $\phi_\mu$ is determined by $\phi_\mu(C_\fs, C_\fs) = 1$, which is nonzero.
For $\lambda$ we have
\[
C^\lambda_{\ft\ft}C^\lambda_{\ft\ft}
= 1 - Q^{-2} + (Q^{-2}+1)Q\inv t \equiv (Q^{-2}+1)C^\lambda_{\ft\ft} \mod A_{<\lambda}.
\]
That  is, $\phi_\lambda$  is determined by $\phi_\mu(C_\ft, C_\ft) = (Q^{-2}+1)$, which can be zero when $f^\B = Q^{-2} + 1 = 0$.
Therefore, $S_{Q,q}^\B(2,1)$ is not quasi-hereditary in an explicit way. 
\endexa

One can also see that $S_{Q,q}^\B(2,1)$ is not quasi-hereditary because 
if it were then it would have finite global dimension. However, $\cH_{Q^{-1}}^\A(\Sigma_2)$ is a Frobenius algebra with infinite global dimension.  
\begin{conj}\label{conj:QHB}
The algebra $S_{Q,q}^\B(n,d)$ is quasi-hereditary if and only if $f^\B_d(Q,q)$ is invertible.
\end{conj}
\section{Representation Type} \label{sec:Rep}
%=============================================

\subsection{} Let $A$ be a finite-dimensional algebra over a field $K$. A fundamental question one can ask about $A$ is to describe its representation type. 
The algebra $A$ is {\em semisimple} if and only if every finite-dimensional module (i.e., $M\in \text{mod}(A)$) is a direct sum of simple modules. 
This means that indecomposable modules for $A$ are simple. If $A$ admits finitely many finite-dimensional indecomposable modules, $A$ is 
said to be of {\em finite representation type}. If $A$ does not have finite representation type $A$ is of {\em infinite representation type}. 

A deep theorem of Drozd states that finite dimensional algebras of infinite representation type can be split into two mutually exclusive categories: tame or wild. 
An algebra $A$ has {\em tame representation type} if for each dimension there exists finitely many one-parameter families of indecomposable objects in $\text{mod}(A)$. 
The indecomposable modules for algebras of tame representation type are classifiable. On the other hand, the algebras of {\em wild representation type} are those whose 
representation theory is as difficult to study as the representation theory of the free associative algebra $k\langle x, y \rangle$ on two variables.  Classifying the finite-dimensional $k\langle x,y \rangle$-modules 
is very much an open question. 

\subsection{Summary: Type A Results} \label{S:Type A-reptype} The following results from \cite[Theorem 1.3(A)-(C)]{EN01} summarize 
the representation type for the $\bar{q}$-Schur algebra for type $\A$ over $K$. Assume that $p = \textup{char}(K)$, $\bar{q} \in K^{\times}$ has multiplicative order $l$ and $\bar{q} \neq 1$. 

\begin{thm} \label{T:TypeAss}
The algebra $S^{\A}_{\bar{q}}(n,r)$ is semisimple if and only if one of the following holds:
\begin{itemize}
\item[(i)]  $n=1$;
\item[(i)]  $\bar{q}$ is not a root of unity;
\item[(ii)]  $\bar{q}$ is a primitive $l$th root of unity and $r<l$;
\item[(iii)] $n=2$, $p=0$, $l=2$ and $r$ is odd;
\item[(iv)] $n=2$, $p\geq 3$, $l=2$ and $r$ is odd with $r < 2p+1$.
\end{itemize}
\end{thm}

\begin{thm} \label{T:finitetypeA} 
The algebra $S^{\A}_{\bar{q}}(n,r)$  has finite representation type but
is not semi-simple
if and only if $\bar{q}$ is a primitive $l$th root of unity
with $l \leq  r$,
and one of the following holds:
\begin{itemize}
\item[(i)]  $n \geq 3$ and $r < 2l$;
\item[(ii)] $n=2$, $p\neq 0$, $l\geq 3$ and  $r<lp$;
\item[(iii)] $n=2$, $p=0$ and either $l\geq 3$,  or $l=2$ and $r$ is even;
\item[(iv)] $n=2$, $p \geq 3$, $l=2$ and $r$ even with $r<2p$, or
$r$ is odd with $2p+1 \leq r < 2p^2+1$.
\end{itemize}
\end{thm}

\begin{thm}  
The algebra $S^{\A}_{\bar{q}}(n,r)$  has tame representation type
if and only if $\bar{q}$ is a primitive $l$th root of unity  and one of the following holds:
\begin{itemize}
\item[(i)] $n=3$, $l=3$,  $p\neq 2$ and $r=7,8$;
\item[(ii)] $n=3$, $l=2$ and $r=4,5$;
\item[(iii)] $n=4$, $l=2$ and $r=5$;
\item[(iv)] $n=2$, $l\geq 3$, $p=2$ or $p=3$ and $pl\leq r < (p+1)l$;
\item[(v)] $n=2$, $l=2$, $p=3$ and $r\in \{ 6,  19, 21, 23 \}$.
\end{itemize}
\end{thm}

\subsection{} In this section we summarize some of the fundamental results that are used to 
classify the representation type of Schur algebras. The first proposition can be verified by using the existence of the determinant representation for 
$S^{\A}_{\bar{q}}(n,r_{1})$ (cf. \cite[Proposition 2.4B]{EN01}).

\begin{prop} If $S^{\A}_{\bar{q}}(n,r_{1})\otimes S^{\A}_{\bar{q}}(n,r_{2})$ has wild representation type then  
$S_{\bar{q}}^{\A}(n,r_{1}+n)\otimes S_{\bar{q}}^{\A}(n,r_{2})$ has wild representation type. 
\end{prop} 

Next we can present a sufficient criteria to show that the tensor product of type A Schur algebras has wild representation type. 

\begin{prop} \label{P:prodwildreptype} Suppose that the Schur algebras $S^{\A}_{\bar{q}}(n,r_{1})$ and $S^{\A}_{\bar{q}}(n,r_{2})$ are non-semisimple algebras. 
Then $S^{\A}_{\bar{q}}(n,r_{1})\otimes S^{\A}_{\bar{q}}(n,r_{2})$ has wild representation type. 
\end{prop} 

\begin{proof} First note that $S^{\A}_{\bar{q}}(n,r)$ is a quasi hereditary algebra and if $S^{\A}_{\bar{q}}(n,r)$ is not semisimple then it must have 
a block with at least two simple modules. 

Suppose that $S_{1}, S_{2}, S_{3}$ are three simple modules in $S^{\A}_{\bar{q}}(n,r_{1})$ with 
$\text{Ext}^{1}_{S^{\A}_{\bar{q}}(n,r_{1})}(S_{1},S_{2})\neq 0$ and $\text{Ext}^{1}_{S^{\A}_{\bar{q}}(n,r_{1})}(S_{2},S_{3})\neq 0$. 
Note that via the existence of the transposed duality, 
$$\text{Ext}^{1}_{S^{\A}_{\bar{q}}(n,r_{1})}(S_{i},S_{j})\cong \text{Ext}^{1}_{S^{\A}_{\bar{q}}(n,r_{1})}(S_{j},S_{i})$$ 
for $i,j=1,2,3$. Similarly, let $T_{1},T_{2}$ be two simple modules for $S^{\A}_{\bar{q}}(n,r_{2})$ with $\text{Ext}^{1}_{S^{\A}_{\bar{q}}(n,r_{2})}(T_{1},T_{2})\neq 0$. 
Then the $\text{Ext}^{1}$-quiver for $S^{\A}_{\bar{q}}(n,r_{1})\otimes S^{\A}_{\bar{q}}(n,r_{2})$ will have a subquiver of the 
form as in Figure~\ref{A4} below.
This quiver cannot be separated into a union of Dynkin diagrams or extended Dynkin diagrams. Consequently, 
$S^{\A}_{\bar{q}}(n,r_{1})\otimes S^{\A}_{\bar{q}}(n,r_{2})$ must has wild representation type.  

\begin{figure}[ht]
\begin{center}
\setlength{\unitlength}{.5cm}
\begin{picture}(12,4)
\put(2.0,0.0){$\bullet$}
\put(6.0,0.0){$\bullet$}
\put(10.0,0.0){$\bullet$}
\put(2.0,4.0){$\bullet$}
\put(6.0,4.0){$\bullet$}
\put(10.0,4.0){$\bullet$}

\put(2.8,0.2){\vector(1,0){2.5}}
\put(5.2,-0.2){\vector(-1,0){2.5}}

\put(6.8,0.2){\vector(1,0){2.5}}
\put(9.2,-0.2){\vector(-1,0){2.5}}

\put(2.8,4.2){\vector(1,0){2.5}}
\put(5.2,3.8){\vector(-1,0){2.5}}

\put(6.8,4.2){\vector(1,0){2.5}}
\put(9.2,3.8){\vector(-1,0){2.5}}

\put(1.8,0.8){\vector(0,1){2.5}}
\put(2.2,3.2){\vector(0,-1){2.5}}

\put(5.8,0.8){\vector(0,1){2.5}}
\put(6.2,3.2){\vector(0,-1){2.5}}

\put(9.8,0.8){\vector(0,1){2.5}}
\put(10.2,3.2){\vector(0,-1){2.5}}
\end{picture}
\caption{}\label{A4}
\end{center}
\end{figure}

The other case to consider is when the blocks of $S^{\A}_{\bar{q}}(n,r_{1})$ and $S^{\A}_{\bar{q}}(n,r_{2})$ have at most two 
simple modules. Let ${\mathcal B}_{j}$ be a block of $S^{\A}_{\bar{q}}(n,r_{j})$ for $j=1,2$ with two simple modules. 
There are four simple modules in ${\mathcal B}_{1}\otimes {\mathcal B}_{2}$ and the structure of the projective 
modules are the same as regular block for category ${\mathcal O}$ for the Lie algebra of type $\A_{1}\times \A_{1}$ 
(cf. \cite[4.2]{FNP01}). The argument in \cite[Lemma 4.2]{FNP01} can be use to show that ${\mathcal B}_{1}\otimes {\mathcal B}_{2}$ 
has wild representation type. 
\end{proof} 

\subsection{} The results in \cite[Theorem 1.3(A)-(C)]{EN01} entail using a different parameter $\bar{q}$ than the parameter 
$q$ in our paper. The relationship is given by $\bar{q}=q^{-2}$ or equivalently $q^{2}=(\bar{q})^{-1}$ with $S_{q}^{\A}(n,d)\cong 
S_{\bar{q}}^{\A}(n,d)$. This means that 
\begin{itemize} 
\item $q$ is generic if and only if $\bar{q}$ is generic, 
\item $q^{2}$ is a primitive $l$th root of unity if and only if $\bar{q}$ is a primitive $l$th root of unity;  
\item if $q$ is a primitive $2s$-th root of unity if and only if  $\bar{q}$ is a primitive $s$-th root of unity;
\item if $q$ is a primitive $(2s+1)$-th root of unity if and only if $\bar{q}$ is a primitive $(2s+1)$-th root of unity. 
\end{itemize} 

Now let $n^{\prime}\geq n$. By Proposition~\ref{prop:idem}, under suitable conditions on $n^{\prime}$ and $n$, there exists an idempotent $e\in S_{Q,q}^{\B}(n^{\prime},d)$ such that 
$S_{Q,q}^{\B}(n,d)\cong eS_{Q,q}^{\B}(n^{\prime},d)e$. By using the proof in \cite[Proposition 2.4B]{EN01}, one has the 
following result. 

\begin{prop} \label{P:increaserank}  Let $n^{\prime}\geq n$ with  $n' \geq n$ and $n' \equiv n \mod 2$.
\begin{itemize} 
\item[(a)] If $S^{\B}_{Q,q}(n,d)$ is not semisimple then $S^{\B}_{Q,q}(n^{\prime},d)$ is not semisimple. 
\item[(b)] If $S^{\B}_{Q,q}(n,d)$ has wild representation type then $S^{\B}_{Q,q}(n^{\prime},d)$ has wild representation type. 
\end{itemize}
\end{prop}

\subsection{Type B Results} Throughout this section, let $S^{\B}_{Q,q}(n,d)$ be the $q$-Schur algebra of Type B 
under the condition that the polynomial $f_{d}^{\B}(Q,q)\neq 0$. Moreover, assume that $q^{2}\neq 1$ (i.e., $q\neq 1$ or a primitive $2$nd root of unity). 
One can apply the isomorphism in Theorem~\ref{thm:SBA} to determine the representation type for $S^{\B}_{Q,q}(n,d)$ from the Type A results stated in Section~\ref{S:Type A-reptype}. 

\begin{thm} The algebra $S^{\B}_{Q,q}(n,d)$ is semisimple if and only if one of the following holds:
\begin{itemize}
\item[(i)] $n=1$; 
\item[(ii)]  $q$ is not a root of unity;
\item[(iii)]  $q^{2}$ is a primitive $l$th root of unity and $d<l$;
\item[(iv)] $n=2$ and $d$ arbitrary;
\end{itemize}
\end{thm}

\begin{proof} The semisimplicity of (i)-(iii) follow by using Theorem~\ref{thm:SBA} with Theorem~\ref{T:TypeAss}. 
The semisimplicity of (iv) follows by Theorem~\ref{thm:SBA} and the fact that $S_{q}^{\A}(1,d)$ is always semisimple. 

Now assume that $q^{2}$ is a primitive $l$th root of unity, $d\geq l$, $n\geq 3$ and $l\geq 3$. Consider the case 
when $n=3$. From Theorem~\ref{thm:SBA}, 
\begin{equation} \label{eq:n=3iso}
S_{Q,q}^{\B}(3,d)\cong \bigoplus_{i=0}^{d} S_{q}^{\A}(2,i)\otimes S_{q}^{\A}(1,d-i).
\end{equation} 
If $d\geq l$ then $S_{q}^\A(2,l)$ appears as a summand of $S_{Q,q}^{\B}(3,d)$ (when $i=d-l$). For $l\geq 3$,  
$S^\A_{q}(2,l) \simeq S^{A}_{\bar{q}}(2,l)$ is not semisimple. It follows that $S_{Q,q}^{\B}(3,d)$ is not semisimple for $d\geq l$. 
One can repeat the same argument for $n=4$ to show that $S_{Q,q}^{\B}(4,d)$ is not semisimple for $d\geq l$. 
Now apply Proposition~\ref{P:increaserank}(a) to deduce that $S_{Q,q}^{\B}(n,d)$ is not semisimple 
for $n\geq 3$ and $d\geq l$. 
\end{proof} 

\begin{thm} The algebra $S^{\B}_{Q,q}(n,d)$  has finite representation type but
is not semisimple if and only if $q^{2}$ is a primitive $l$th root of unity
with $l \leq  d$,
and one of the following holds:
\begin{itemize}
\item[(i)] $n\geq 5$, $l\leq d < 2l$; 
\item[(ii)]  $n=3$, $p=0$ and $l \leq d$; 
\item[(iii)] $n=3$, $p \geq 2$ and $l \leq d < l p$; 
\item[(iv)] $n=4$, $p=0$, $l=2$ and $d\geq 4$ with $d$ odd.  
\item[(v)] $n=4$, $p\geq 3$, $l=2$ and $4< d\leq 2p-1$ with $d$ odd.  
\end{itemize}
The algebra $S^{\B}_{Q,q}(n,d)$  has tame representation type if and only if 
\begin{itemize} 
\item[(vi)] $n=3$, $l=2$, $p=3$ and $d=6$; 
\item[(vii)] $n=3$, $l\geq 3$, $p=2$ or $3$ and $l p \leq d < l(p+1)$; 
\item[(viii)] $n=4$, $l=2$, $p=3$ and $d=7$. 
\end{itemize} 
\end{thm}

\begin{proof} We first reduce our analysis to the situation where $n=3$ and $4$. 
Assume that $n \geq 5$ so $\nhc \geq 3$ and $\nhf \geq 2$.
By Theorem~\ref{T:TypeAss}, the algebras $S^{\A}_{q}(2,l)$ and $S^{\A}_{q}(i,l+j)$ are not semisimple for $i \geq 3, j \geq 0$, and hence neither are $S^{\A}_{q}(\nhc,l+j)$ and $S^\A_q(\nhf, l)$ for $n \geq 5, j \geq 0$.
Therefore,$S^{\A}_{q}(\nhc,l+j) \otimes S^\A_q(\nhf, l)$ has wild representation type by Proposition~\ref{P:prodwildreptype}. 
It follows that $S^{\B}_{Q,q}(n, d)$ has wild representation type for $d\geq 2l, n \geq 5$. 
When $l\leq d < 2l$ and $n\geq 5$, one can use Theorem~\ref{thm:SBA} in conjunction with Theorem~\ref{T:finitetypeA} to prove that 
$S^{\B}_{Q,q}(n,d)$ has finite representation type. 

Now consider the case when $n=3$. The isomorphism (\ref{eq:n=3iso}) indicates that we can reduce our analysis to considering 
$S^{\A}_{q}(2,r)$. From this isomorphism and Theorem~\ref{T:finitetypeA}, one can verify that (i) when $\text{char } K=0$ then 
$S^{\B}_{Q,q}(3,d)$ has finite representation type (but is not semisimple) for $l\leq d$, (ii) when $\text{char } K=p>0$ then 
$S^{\B}_{Q,q}(3,d)$ has  finite representation type (but is not semisimple) for $l\leq d <l p$, and (iii) when $\text{char } K=p>0$, 
$S^{\B}_{Q,q}(3,d)$ has  infinite representation type for $d\geq l p$. 

For $n=3$, one can also see that under conditions (vi) and (vii), $S^{\B}_{Q,q}(3,d)$ has tame representation type. Moreover, 
one can verify that $S^{\B}_{Q,q}(3,d)$ has wild representation type in the various complementary cases. 

Finally let $n=4$. From Proposition~\ref{P:prodwildreptype}, $S^{\A}_{q}(2,l)\otimes S^{\A}_{q}(2,l)$ and 
$S^{\A}_{q}(2,l)\otimes S^{\A}_{q}(2,l+1)$ has wild representation type for $l\geq 3$. Therefore, 
$S^{\B}_{Q,q}(4,d)$ has wild representation type for $d\geq 2l$ and $l\geq 3$. For $l=2$, the same 
argument can be used to show that $S^{\B}_{Q,q}(4,d)$ has wild representation type for $d$-even and $d\geq 4$. 

This reduces us to analyzing $S^{\B}_{Q,q}(4,d)$ when $l=2$ and $d\geq 4$ is odd. By analyzing the 
components of $S^{\B}_{Q,q}(4,d)$ via the isomorphism in Theorem~\ref{thm:SBA}, one can show that for $d$ odd: (i) $S^{\B}_{Q,q}(4,d)$ 
has finite representation type (not semisimple) for $4\leq d\leq 2p-1$ and $p\geq 3$, (ii) $S^{\B}_{Q,q}(4,d)$ 
has finite representation type (not semisimple) for $d\geq 4$ and $p=0$, (iii) $S^{\B}_{Q,q}(4,d)$ 
has wild representation type for $d\geq 2p+1$ for $p\geq 5$, and (iv) $S^{\B}_{Q,q}(4,d)$ 
has wild representation type for $d\geq 2p+3$ for $p=3$. One has then show that $S^{\B}_{Q,q}(4,7)$ for $p=3$, $l=2$ has 
tame representation type since the component $S^{\A}_{q}(2,6)\otimes S^{\A}_{q}(2,1)$ has tame representation type and 
the remaining components have finite representation type.

\end{proof} 

Note that for the case $\bar{q}=1$ (i.e., $q^{2}=1$) one obtains the classical Schur algebra for type $A$, one can use the results in \cite{E93} \cite{DN98} \cite{DEMN99} to obtain 
classification results in this case for $S_{Q,q}^{B}(n,d)$.

%=============================================
\section{Quasi-hereditary covers}
\label{sec:qhcover}
%=============================================
In this section we first recall results on 1-faithful quasi-hereditary covers due to Rouquier \cite{Ro08}. Then we demonstrate that our Schur algebra is a 1-faithful quasi-hereditary cover of the type B Hecke algebra via Theorem~\ref{thm:SBA}. Hence, it module category identifies the category $\cO$ for the rational Cherednik algebra of type B, see Theorem~\ref{thm:O}. A comparison of our Schur algebra with Rouquier's Schur-type algebra is also provided.
%=============================================
\subsection{1-faithful covers}
Let $\mathcal{C}$ be a category equivalent to the module category of a finite dimensional projective $K$-algebra $A$, and let $\Delta = \{\Delta(\lambda)\}_{\lambda \in \Lambda}$ be a set of objects of $\mathcal{C}$ indexed by an interval-finite poset structure $\Lambda$.
Following \cite{Ro08}, we say that $\mathcal{C}$ (or $(\mathcal{C}, \Delta)$) is a {\em highest weight category} if the following conditions are satisfied:
\enu
\item[(H1)] $\End_{\mathcal{C}}(\Delta(\lambda)) = K$ for all $\lambda \in \Lambda$;
\item[(H2)] If $\Hom_{\mathcal{C}}(\Delta(\lambda) , \Delta(\mu)) \neq 0$ then $\lambda \leq \mu$;
\item[(H3)] If $\Hom_{\mathcal{C}}(\Delta(\lambda), M) = 0$ for all $\lambda \in \Lambda$ then $M = 0$;
\item[(H4)] For each $\Delta(\lambda) \in \Delta$ there is a projective module $P(\lambda) \in \mathcal{C}$ such that ker$(P(\lambda) \to \Delta(\lambda))$ has a $\Delta$-filtration, i.e., finite filtrations whose quotients are isomorphic to objects in $\Delta$.
\endenu
%-----------------------------------------------------------------------------------------------
%-----------------------------------------------------------------------------------------------
Let $A$-mod be the category of finitely generated $A$-modules.
The algebra $A$ is called a {\em quasi-hereditary cover} of $B$ if the conditions below hold:
\enu
\item[(C1)] $A$-mod admits a highest weight category structure $(A\textup{-mod}, \Delta)$;
\item[(C2)] $B = \End_A(P)$ for some projective $P \in A$-mod;
\item[(C3)] The restriction of $F=\Hom_A(P,-)$ to the category of finitely generated projective $A$-modules is fully faithful.
\endenu
Quasi-hereditary covers are sometimes called highest weight covers since the notion of highest weight category corresponds to that of split quasi-hereditary algebras \cite[Theorem~4.16]{Ro08}.
We also say that $(A,F)$ is a quasi-hereditary cover of $B$.
Moreover, a category $\mathcal{C}$ (or the pair $(\mathcal{C}, F)$) is said to be a quasi-hereditary cover of $B$ if $\mathcal{C} \simeq A\textup{-mod}$ for some quasi-hereditary cover $(A,F)$ of $B$.

%-----------------------------------------------------------------------------------------------
Following \cite{Ro08}, 
a quasi-hereditary cover $A$ of $B$ is {\em $i$-faithful} if
\eq\label{eq:i-f}
\Ext_A^j(M,N) \simeq \Ext_B^j(FM,FN)
\quad
\textup{for}
\quad
j\leq i,
\endeq
and for all $M,N \in A$-mod admitting $\Delta$-filtrations.
Furthermore, a quasi-hereditary cover $(\mathcal{C}, F)$ of $B$ is said to be $i$-faithful if the diagram below commutes for some quasi-hereditary cover $(A,F')$ of $B$:
\[\begin{tikzcd}
\cC \arrow{rd}[swap]{F} \arrow{rr}{\simeq} &{}& A\textup{-mod} \arrow{ld}{F'} 
\\
{}&B\textup{-mod}
\end{tikzcd}
\]
Rouquier proved in \cite[Theorem~4.49]{Ro08} a uniqueness theorem for the 1-faithful quasi-hereditary covers which we paraphrase below:
%-----------------------------------------------------------------------------------------------
\begin{prop}\label{prop:unicity}
Let $B$ be a finite projective $K$-algebra that is split semisimple,
and let $(\cC_i,F_i)$ for $i=1,2$ be $1$-faithful quasi-hereditary covers of
$B$ with respect to the partial order $\le_i$ on $\Irr(B)$.
If $\le_1$ is a refinement of $\le_2$ then there is an equivalence $\cC_1 \simeq \cC_2$ of quasi-hereditary covers of $B$ inducing the bijection $\Irr(\cC_1)\simeq\Irr(B)\simeq\Irr(\cC_2)$.
\end{prop}
%-----------------------------------------------------------------------------------------------
%-----------------------------------------------------------------------------------------------
\subsection{Rational Cherednik algebras}
Let $(W,S)$ be a finite Coxeter group, and let  $A_W$ be the corresponding rational Cherednik algebra over $\CC[h_u; u\in U]$ as in \cite{Ro08}, where 
$U = \bigsqcup_{s \in S} \{s\}\times \{1, \ldots, e_s\}$ and $e_s$ is the size of the pointwise stabilizer in $W$ of the hyperplane corresponding to $s$.
If $W = W^\B(d)$ and $S = \{s_0, s_1\}$ then $U = \{(s_i,j) ~|~ 0\leq i,j \leq 1\}$. In this case we assume that
\eq\label{def:h}
h_{(s_1,0)} = h,
\quad
h_{(s_1,1)} = 0,
\quad
h_{(s_0,i)} = h_i 
\quad
\textup{for}
\quad
i=0,1.
\endeq

\rmk
In \cite{EG02} the rational Cherednik algebra $\mathbf{H}_{t,c}$ is defined for a parameter $t\in \CC$, and a $W$-equivariant map $c:S \to\CC$. The two algebras, $A_{W}$ and $\mathbf{H}_{t,c}$, coincide if $t=1$, $h_{(s,0)} = 0$ and $h_{(s,1)} = c(s)$ for all $s\in S$.
\endrmk
%, i.e., the quotient of
%$\CC[h_u : u\in U]\otimes_\CC T(V\oplus V^*)\rtimes W$ by the relations
%\[
%[\xi,\eta]=0 \text{ for } \xi,\eta\in V,
%\quad
%[x,y]=0 \text{ for } x,y\in V^*,
%\quad
%[\xi,x]=\langle \xi,x\rangle+\sum_{H\in A_W}
%\frac{\langle\xi,\alpha_H\rangle\langle v_H,x\rangle}
%     {\langle v_H,\alpha_H\rangle} \gamma_H,
%\]
%where
%$\gamma_H=\sum_{w\in W_H-\{e\}} \left(\sum_{j=0}^{e_H-1}\det(w)^{-j}(h_{H,j}-h_{H,j-1})\right)w.$
Let $\cO_W$ be the category of finitely generated $A_W$-modules that are locally nilpotent for $S(V)$.
It is proved in \cite{GGOR03} that $(\cO_W, \Delta_W)$ is a highest weight category of $\cH(W)$-mod
\[
\Delta_W = \{\Delta(E) := A_W \otimes_{S(V) \rtimes W} E ~|~ E \in \Irr(W)\},
\]
See \cite[3.2.1--3]{Ro08} for the partial order $\leq$ on $\Irr(W)$.
Let $\Lambda_2^+(d)$ be the poset of all bipartitions of $d$ on which the dominance order $\unlhd$ is given by 
$\lambda\unlhd\mu$ if, for all $s\ge 0$,
\eq
\sum_{j=1}^s |\lambda_j^{(1)}|\le
\sum_{j=1}^s |\mu_j^{(1)}|
,
\quad
|\lone|+ \sum_{j=1}^s |\lambda_j^{(r)}|\le
|\mu^{(1)}| + \sum_{j=1}^s |\mu_j^{(r)}|.
\endeq
For $\lambda \in \Lambda_2^+(d)$, set
\eq
W^\B_\lambda(d) = C_2^d \rtimes (\Sigma_{\lone} \times \Sigma_{\ltwo}),
\endeq
%For $r = 1, 2$, set
%\eq
%I_\lambda(r)=\{\sum_{i=1}^{r-1}|\lambda^{(i)}|+1,
%\sum_{i=1}^{r-1}|\lambda^{(i)}|+2,\ldots,
%\sum_{i=1}^r|\lambda^{(i)}|\}.
%\endeq
Set
\eq
I_\lambda(1)=\{1, \ldots, |\lone|\},
\quad
I_\lambda(2)=\{|\lone|+1, \ldots, d\}.
\endeq
Following \cite[6.1.1]{Ro08}, there is a bijection 
\eq
\Lambda_2^+(d) \to \Irr(W^\B(d)),
\quad
\lambda = (\lone, \ltwo) \mapsto \chi_\lambda = \textup{Ind}^{W^\B(d)}_{W^\B_\lambda(d)} (\chi_{\lone} \otimes \phi^{(2)} \chi_{\ltwo}),
\endeq
where
$\chi_\lambda$ is the irreducible character of $W^\B(d)$ corresponding to $\lambda$, and $\phi^{(2)}$ is the 1-dimensional character of $C_2^{I_\lambda(2)} \rtimes \Sigma_{I_\lambda(2)}$ whose restriction to $C_2^{I_\lambda(2)}$ is $\det$ and the restriction to $\Sigma_{I_\lambda(2)}$ is trivial.

Rouquier showed that the order $\leq$ is a refinement of the dominance order $\unlhd$ under an assumption on the parameters $h, h_i$'s for the rational Cherednik algebra as follows:
%-----------------------------------------------------------------------------------------------
\begin{lem}\cite[Proposition~6.4]{Ro08} \label{lem:refine}
Assume that $W = W^B(d)$, $h\le 0$ and $h_1-h_{0}\ge (1-d) h$ (see \eqref{def:h}).
Let $\lambda,\mu \in \Lambda_2^+(d)$. 
If $\lambda\unlhd\mu$, then $\chi_\lambda \le \chi_\mu$ on $\Irr(W)$.
\end{lem}
%-----------------------------------------------------------------------------------------------
\rmk
The assumption in Lemma~\ref{lem:refine} on the parameters is equivalent to $c(s_0) = h_1 \ge 0$ using Etingof-Ginzburg's convention.
\endrmk
%\eq
%\chi < \chi' \Leftrightarrow c_\chi - c_{\chi'} \in \ZZ_{>0}.
%\endeq
Let $KZ_W$ be the KZ functor $\cO_W \to \cH(W)$-mod. 
We paraphrase \cite[Theorem~5.3]{Ro08} in our setting as below:
\begin{prop}\label{prop:OW}
If $W = W^\B(d)$ and $\cH(W) = \cH^\B_{Q,q}(d)$, then
$(\cO_W, KZ_W)$ is a quasi-hereditary cover of $\cH(W)$-mod.
Moreover, the cover is 1-faithful if $(q^2+1)(Q^2+1) \neq 0$.
\end{prop}

It is shown in \cite{Ro08} that under suitable assumptions, $\cO_{W^\B(d)}$ is equivalent to the module category of a Schur-type algebra $S^\R(d)$ which does not depend on $n$ using the uniqueness property Proposition~\ref{prop:unicity}. Below we give an interpretation in our setting. 

Let $\Lambda_2(d)$ be the set of all bicompositions of $d$. In \cite{DJM98b} a cyclotomic Schur algebra over $\QQ(q,Q,Q_1,Q_2)$ for each saturated subset $\Lambda \subset \Lambda_2(d)$, which specializes to cyclotomic Schur algebras $S_Q(\Lambda)$ over $K$ is defined (see Section~\ref{sec:CycS}). 
Moreover, in \cite{Ro08} an algebra $S_Q(\Lambda)$ is defined that is Morita equivalent to $S_Q(\Lambda)$ as given below:
\eq
S^\R(d)
:=
\End_{\cH_{Q,q}^\B(d)} (P_d),
\quad
P_d :=
\bigoplus_{\lambda\in \Lambda_2^+(d)} m_\lambda \cH_{Q,q}^\B(d).
\endeq
where $m_\lambda$ is defined in \eqref{def:mlambda}.
Note that $S^\R(d)$ does not depend on $n$. Set
\eq
F^\R_d = \Hom_{S^\R(d)}(  P_d , -): S^\R(d)\textup{-mod} \to \cH_{Q,q}^\B(d)\textup{-mod}.
\endeq
\begin{prop} \cite[Theorem~6.6]{Ro08} \label{prop:SchurR}
\enu[(a)]
\item
The category $\Mod(S^\R(d))$ is a highest weight category for the dominance order;
\item
$(S^\R(d), F^\R_d)$ is a quasi-hereditary cover of $\cH_{Q,q}^\B(d)$;
\item 
The cover $(S^\R(d), F^\R_d)$  is 1-faithful if
\eq\label{eq:condR}
(q^2+1)(Q^2+1)\neq 0,
\quad
\textup{and}
\quad
f^\B_{Q,q}(d) \cdot \prod_{i=1}^d (1+q^2+\cdots+q^{2(i-1)})
\neq 0.
\endeq
\endenu
\end{prop}
The category $\cO$ for the type B rational Cherednik algebra together with its KZ functor can then be identified by
combining Propositions~\ref{prop:unicity}, \ref{prop:OW} and \ref{prop:SchurR}.
In other words, the following diagram commutes if \eqref{eq:condR} holds:
\[\begin{tikzcd}
\cO_{W^\B(d)} \arrow{rd}[swap]{KZ_{W^\B(d)}} \arrow{rr}{\simeq} &{}& S^\R(d)\textup{-mod} \arrow{ld}{F^\R_d} 
\\
{}&\cH_{Q,q}^\B(d)\textup{-mod}
\end{tikzcd}
\]
%-----------------------------------------------------------------------------------------------
\subsection{1-faithfulness of $S^\B_{Q,q}(n,d)$-mod}
%-----------------------------------------------------------------------------------------------
Let $\ell$ be the multiplicative order of $q^2$ in $K^\times$. 
In this section we use the following assumptions:
\eq\label{eq:condLNX}
f^\B_d(Q,q) = \prod_{i=1-d}^{d-1}(Q^{-2}+q^{2i}) \in K^\times,
\quad
r := \nhf \geq d,
\quad
\ell \geq 4.
\endeq
As a consequence, there exists a type B Schur functor by Proposition~\ref{prop:SFB}.
For type A, it is known in \cite{HN04} that the $q$-Schur algebra is a 1-faithful quasi-hereditary cover of the type A Hecke algebra if $\ell \geq 4$.
Moreover, Theorem~\ref{thm:SBA} applies and hence we will see shortly that $S^\B_{Q,q}(n,d)$ is a 1-faithful quasi-hereditary cover of $\cH^\B_{Q,q}(d)$.
Furthermore, Proposition~\ref{prop:unicity} implies that we have a concrete realization for the category $\cO$ for the type B rational Cherednik algebra together with its KZ functor using our Schur algebra.
%Throughout this section we let $A = S^\B_{Q,q}(n,d)$, $B = \cH^\B_{Q,q}(d)$.

%-----------------------------------------------------------------------------------------------
\cor\label{cor:SBHWT}
If $f^\B_d\in K^\times$, then
$ S^\B_{Q,q}(n,d)$-mod is a highest weight category.
\endcor
%-----------------------------------------------------------------------------------------------
\proof
It follows immediately from the isomorphism with the direct sum of type A $q$-Schur algebras that $ S^\B_{Q,q}(n,d)$-mod is a highest weight category.
\endproof
%-----------------------------------------------------------------------------------------------
In below we characterize a partial order for highest weight category $S^\B_{Q,q}(n,d)$-mod  
obtained via Corollary~\ref{cor:SBHWT} and the dominance order for type A.
Denote the set of all $N$-step partitions of $D$ by $\Lambda^\A(N,D)$. Set
\eq
\Delta^\A_{N,D} = \{\Delta^\A(\lambda)~|~ \lambda \in \Lambda^\A(N,D)\}.
\endeq 
Now $\Delta^\A_{N,D}$ is a poset with respect to the dominance order $\unlhd$ on $\Lambda^\A(N,D)$. 
It is well known that for all non-negative integers $N$ and $D$, $(S^\A_q(N,D)\textup{-mod}, \Delta^\A_{N,D})$ is a highest weight category.

Recall $\cF_S$ from \eqref{def:FS} and $\Lambda^\B(n,d)$ from \eqref{def:LB}.
Set 
\eq\label{def:DB}
\Delta^\B_{n,d} = \{ 
\Delta^\B(\lambda) := \cF^{-1}(\Delta^\A(\lone) \otimes \Delta^\A(\ltwo) )
~|~
\lambda=(\lone, \ltwo)  \in \Lambda^B(n,d)
\}.
\endeq
Now $\Delta^\B_{n,d}$ is a poset with respect to the dominance order (also denoted by $\unlhd$) on $\Lambda^B(n,d) \subset \Lambda_2^+(d)$.
Hence,  $(S^\B_q(n,d)\textup{-mod}, \unlhd)$ is a highest weight category.

%-----------------------------------------------------------------------------------------------
%\begin{lem}[\cite{DEN04}] \label{lem:SS}
%For an $A$-module $M$ and a $B$-module $N$, there is a first quadrant spectral sequence
%\[
%E_2^{i+j} = \Ext_A^i(M, R^jG(N)) \Rightarrow \Ext^{i+j}_{eAe} (eM, N).
%\]
%\end{lem}
%-----------------------------------------------------------------------------------------------
\lem\label{lem:1}
Assume that $S^\B_{Q,q}(n,d)$ is a quasi-hereditary cover of $\cH^\B_{Q,q}(d)$.
If \eqref{eq:condLNX} holds, then the cover is 1-faithful.
\endlem
%-----------------------------------------------------------------------------------------------
\proof
Write $A=S^\B_{Q,q}(n,d), B = \cH^\B_{Q,q}(d), S' = S^\A_q(\nhc, i), S'' = S^\A_q(\nhf, d-i)$ for short. We need to show that,
for all $M, N$ admitting $\Delta^\B$-filtrations,
\[
\Ext^i_A(M, N) \simeq \Ext^i_{eAe}(F^\B_{n,d} M, F^\B_{n,d} N),
\quad
i \leq 1.
\]
Recall $\cF_S$ from \eqref{def:FS}.
Write $\cF_S M = \bigoplus_i M'_i \otimes M''_i$ and $\cF_S N = \bigoplus_i N'_i \otimes N''_i$ for some $M'_i, N'_i \in \textup{Mod}(S')$ and $M''_i, N''_i \in \textup{Mod}(S'')$.
From construction we see that all $M'_i, M''_i, N'_i, N''_i$ admit $\Delta^\A$-filtrations since  $M, N$ have $\Delta^B$-filtrations.

For $\nhf \geq d \geq i \geq 0$, we abbreviate the type A Schur functors  (see \eqref{def:SFA}) by $F' = F^\A_{\nhc,i}, F'' = F^\A_{\nhf, d-i}$.
Since the type A $q$-Schur algebras are 1-faithful provided $\ell \geq 4$, for $j\leq 1$ we have 
\eq
\begin{split}
\Ext^j_{S'}(M'_i, N'_i) &\simeq \Ext^j_{
%e^\A_{\nhc,i} S^\A_q(\nhc,i)e^\A_{\nhc,i}
\cH_{q}(\Sigma_{i+1})
}(F' M'_i, F' N'_i),
\\
\Ext^j_{S''}(M''_i, N''_i) &\simeq \Ext^j_{
%e^\A_{\nhf,d-i}  S^\A_q(\nhf,d-i)e^\A_{\nhf,d-i}
\cH_{q}(\Sigma_{d-i+1})
}(F'' M''_i, F'' N''_i).
\end{split}
\endeq
We show first it is 0-faithful. We have
\eq
\begin{split}
\Hom_{A}(M,N)
&\simeq 
\Hom_{\bigoplus_{i=0}^d S' \otimes S''}
\Big(
\cF_S M, \cF_S N
\Big)
\\
&\simeq \bigoplus_{i=0}^d 
\Hom_{S'}(M'_i,N'_i)
\otimes
\Hom_{S''}(M''_i,N''_i)
\\
&\simeq  \bigoplus_{i=0}^d 
\Hom_{\cH_q(\Sigma_{i+1})}(F' M'_i, F' N'_i)
\otimes
\Hom_{\cH_q(\Sigma_{d-i+1})}(F''  M''_i, F'' N''_i)
\\
&\simeq \bigoplus_{i=0}^d
\Hom_{\cH_q(\Sigma_{i+1}) \otimes \cH_q(\Sigma_{d-i+1})}
\Big(
F' M'_i \otimes  F'  M''_i, F'' N'_i \otimes  F'' N''_i
\Big)
\\
&\simeq  \bigoplus_{i=0}^d
\Hom_{\cH_q(\Sigma_{i+1}) \otimes \cH_q(\Sigma_{d-i+1})}
(
\cF_H F^\flat_{n,d} M, \cF_H F^\flat_{n,d} N
)
\\
&\simeq \Hom_B (F^\flat_{n,d} M, F^\flat_{n,d} N).
\end{split}
\endeq
Note that the second last isomorphism follows from Proposition~\ref{prop:SF}.
For 1-faithfulness, we have
\eq
\begin{split}
\Ext^1_{A}(M,N)
&\simeq \bigoplus_{i=0}^d
\big(
(\Ext^1_{S'}(M'_i,N'_i)
\otimes
\Hom_{S''}(M''_i,N''_i))
\\
&\quad\qquad \oplus
(\Hom_{S'}(M'_i,N'_i)
\otimes
\Ext^1_{S''}(M''_i,N''_i))
\big)
\\
&\simeq \bigoplus_{i=0}^d
\big(
(\Ext^1_{\cH_q(\Sigma_{i+1})}(F' M'_i,F' N'_i)
\otimes
\Hom_{\cH_q(\Sigma_{d-i+1})}(F''  M''_i,F'' N''_i))
\\
&\quad\qquad \oplus
(\Hom_{\cH_q(\Sigma_{i+1})}(F' M'_i,F' N'_i)
\otimes
\Ext^1_{\cH_q(\Sigma_{d-i+1})}(F''  M''_i,F'' N''_i))
\\
&\simeq \bigoplus_{i=0}^d
\Ext^1_{\cH_q(\Sigma_{i+1}) \otimes \cH_q(\Sigma_{d-i+1})}
(
\cF_H F^\flat_{n,d} M, \cF_H F^\flat_{n,d} N
)
\\
&\simeq \Ext^1_B (F^\flat_{n,d} M, F^\flat_{n,d} N).
\end{split}
\endeq

%Applying Lemma~\ref{lem:SS} with $M = P_1$ and $N =eP_2$, we get
%\[
%\Hom_{eAe}(eP_1, eP_2) \simeq \Hom_A(P_1, G(eP_2)).
%\]
%It suffices to check that $G(eP_2) = P_2$.
%Since $G(eP_2) = \Hom_{eAe}(eA, eP_2)$ and $A = \Hom_{eAe}(eA, eA)$, $G(eP_2) $ is projective.
%
%Need $eP \neq 0$ if $P$ is projective.
%
%Need $eP_2$ is indecomposable if $P_2$ is indecomposable.
%
%For 1-faithfulness, we apply Lemma~\ref{lem:SS} and we also need that $R^1G(eP_2) = 0$.
\endproof
%-----------------------------------------------------------------------------------------------
\thm\label{thm:O}
Assume that $W = W^B(d)$, $h\le 0$, $h_1-h_{0}\ge (1-d) h$ (see \eqref{def:h}) and $(q^2+1)(Q^2+1) \in K^\times$.
If \eqref{eq:condLNX} holds, 
then there is an equivalence $\cO_{W} \simeq S^\B_{Q,q}(n,d)$-mod of quasi-hereditary covers.
In other words, the following diagram commutes:
\[\begin{tikzcd}
\cO_{W} \arrow{rd}[swap]{KZ_{W}} \arrow{rr}{\simeq} &{}& S^\B_{Q,q}(n,d)\textup{-mod} \arrow{ld}{F^\flat_{n,d}} 
\\
{}&\cH_{Q,q}^\B(d)\textup{-mod}
\end{tikzcd}
\]
\endthm
%-----------------------------------------------------------------------------------------------
\proof
The theorem follows by combining Propositions~\ref{prop:unicity}, \ref{prop:OW} and Lemmas~\ref{lem:refine}, ~\ref{lem:1}.
\endproof
%-----------------------------------------------------------------------------------------------
\rmk\label{rmk:SR}
The uniqueness theorem for 1-faithful quasi-hereditary covers also applies on our Schur algberas and Rouquier's Schur-type algebras. That is,  the following diagram commutes provided \eqref{eq:condR} and \eqref{eq:condLNX} hold:
\[\begin{tikzcd}
S^\R(d)\textup{-mod} \arrow{rd}[swap]{F^\R_d} \arrow{rr}{\simeq} &{}& S^\B_{Q,q}(n,d)\textup{-mod} \arrow{ld}{F^\flat_{n,d}} 
\\
{}&\cH_{Q,q}^\B(d)\textup{-mod}
\end{tikzcd}
\]
\endrmk
%=============================================

%=============================================
\section{Variants of $q$-Schur algebras of type B/C}
\label{sec:SB}
%=============================================
It is interesting that the type A $q$-Schur algebra admits quite a few distinct generalizations  in type B/C in literature. 
This is due to that the type A $q$-Schur algebra can be realized differently due to the following realizations of the tensor space $(K^n)^{\otimes d}$:
(1) A combinatorial realization as a quantized permutation module (cf. \cite{DJ89});
(2) A geometric realization as the convolution algebra on $\GL_n$-invariant pairs consisting of a $n$-step partial flag and a complete flag over finite field (cf. \cite{BLM90}).

In the following sections we provide a list of $q$-Schur duality/algebras of type B/C in literature, paraphrased so that they are all over $K$, and with only one parameter $q$.
These algebras are all of the form $\End_{\cH^\B_q(d)}(V^{\otimes d})$ for some tensor space that may have a realization $V^{\otimes d} \simeq \bigoplus_{\lambda \in I} M^\lambda$ via induced modules. 
Considering the specialization at $q=1$, we have
\[
\left.M^\lambda\right|_{q=1} = \ind_{H_\lambda}^{W^\B(d)} U,
\quad
H_\lambda \leq W^\B(d) 
\textup{ is a subgroup},
\quad
U \textup{ is usually the trivial module}.
\]
We summarize the properties of the $q$-Schur algebras in the following table:
%\[
%\begin{tabular}{|c|c|c|c|c}
%\hline &cellularity&quasi-heredity&Schur functor
%\\
%\hline
%Coideal $q$-Schur Algebra $S_q^{\B}(n,d)$&new [LNX]&new [LNX]&new [LNX]
%\\
%Cyclotomic Schur Algebra $S_q(\Lambda)$&known \cite{DJM98}&known \cite{DJM98}&unknown
%\\
%Sakamoto-Shoji Algebra $S^{\B}_q(a,b,d)$&unknown&unknown&unknown
%\\
%\hline
%\end{tabular}
%\]
\[
\begin{tabular}{|c|c|c|c|c}
\hline &Coideal $q$-Schur Algebra &Cyclotomic Schur Algebra&Sakamoto-Shoji Algebra
\\
&$S_q^{\B}(n,d)$& $S_q(\Lambda)$& $S^{\B}_q(a,b,d)$
\\
\hline
Index set $I$&
compositions&bicompositions&unclear\\
& $\lambda = (\lambda_i)_{i\in I(n)}$
&
 $\lambda = (\lone, \ltwo)$
&

\\
&with constraints on $\lambda_i$&&
\\
Subgroup $H_\lambda$
&
$W^{\B}(\lambda_0)\times \Sigma_{(\lambda_1, \ldots \lambda_r)}$
& 
$(C_2^{|\lone|} \times C_2^{|\ltwo|})\rtimes \Sigma_{\lambda}$ 
& unknown
\\
Module $U$& trivial&nontrivial&
\\
Schur duality&$(U_q^\B(n), \cH^\B_q(d))$&partial&$(U_q(\fgl_a\times \fgl_b), \cH^\B_q(d))$
\\
Cellularity&new [LNX]&known \cite{DJM98b}&unknown
\\
Quasi-heredity&new [LNX]&known \cite{DJM98b}&unknown
\\
Schur functor&new [LNX]&known \cite{JM00} &unknown
\\
1-faithful cover &new [LNX]&unknown &unknown
\\
\hline
\end{tabular}
\]

For completeness a more involved $q$-Schur algebra (referred as the $q$-Schur$^2$ algebras) of type B is studied in \cite{DS00}.
We also distinguish the coideal $q$-Schur algebras from the slim cyclotomic Schur algebras constructed in \cite{DDY18}.

%-----------------------------------------------------------------------------------------------
\subsection{The coideal Schur algebra $S^{\B}_q(n,d)$} 
%-----------------------------------------------------------------------------------------------
This is the main object in this paper which we have been calling the $q$-Schur algebra of type B.
To distinguish it from the other variants we call them for now the coideal Schur algebras since they are homomorphic images of coideal subalgebras.

For the equal-parameter case, a geometric Schur duality is established between $\cH^\B_q(d)$ and the coideal subalgebra $U_q^{\B}(n)$ as below (cf. \cite{BKLW}):
\[
\Ba{{cclccccc}
U_q^\B(n)
\\
\downarrow
\\
S_q^{\B}(n,d) &\crr&  T_\geo^\B(n,d) \simeq (K^n)^{\otimes d}\simeq T_\alg^\B(n,d)& \crl& \cH^\B_q(d)
}
\]
Note that a construction using type C flags is also available, and it produces isomorphic Schur algebras and hence coideals.
A combinatorial realization $T_\alg^\B(n,d)$ as a quantized permutation module is also available along the line of Dipper-James. 

For the case with two parameters, the algebra $S^{\B}_{Q,q}(n,d)$, when $n$ is even, was first introduced by Green and it is called the hyperoctahedral $q$-Schur algebra \cite{Gr97}.
A two-parameter upgrade for the picture above is partially available -- 
a Schur duality is obtained in \cite{BWW18} between the two-parameter Hecke algebra $\cH^\B_{Q,q}(d)$ and the two-parameter coideal $\mathbb{U}^\B_n$ over the tensor space $\QQ(Q,q)$; a two-parameter upgrade for $T_\alg^\B(n,d)$ is studied in \cite{LL18} -- while a two-parameter upgrade for $T_\geo^\B(n,d)$ remains unknown since dimension counting over finite fields does not generalize to two parameters naively.

To our knowledge, this is the only $q$-Schur algebras for the Hecke algebras of type B that admit a coordinate algebra type construction and a notion of the Schur functors with the existence of appropriate idempotents.

%=============================================
\subsection{Cyclotomic Schur algebras}\label{sec:CycS}
%=============================================
The readers will be reminded shortly that the cyclotomic Hecke algebra $\bbH(r,1,d)$ of type $G(r,1,d)$ is isomorphic to $\cH^\B_q(d)$ at certain specialization when $r=2$.
For each saturated subset $\Lambda$ of the set of all bicompositions, Dipper-James-Mathas (cf. \cite{DJM98b}) define the cyclotomic Schur algebra $\bbS(\Lambda)$: 
\[
S_q(\Lambda) = \End_{\cH^\B_q(d)} T(\Lambda),
\]
where $T(\Lambda)$ is a quantized permutation module that has no known identification with a tensor space.
This generalizes the $(Q,q)$-Schur algebras introduced in \cite{DJM98a}, which is the special case when $\Lambda$ is the set of all bicompositions and $r=2$.

While a cellular structure (and hence a quasi-heredity) is obtained for $S_q(\Lambda)$, it is unclear if it has an analogue of full Schur duality. 
%(\red{see \cite{?} for a double centralizer property with no analogues of the quantum groups}).

We also remark that there is no known identification of $T_\alg^\B(n,d)$ with a $T(\Lambda)$ for some $\Lambda$.

Let $R = \QQ(q,Q, Q_1, Q_2)$. The cyclotomic Hecke algebra (or Ariki-Koike algebra) $\bbH =\bbH(2,1,d)$ is the $R$-algebra generated by $T_0^{\Delta}, \ldots, T_{d-1}^{\Delta}$ subject to the relations below, for $1\leq i \leq d-1, 0 \leq j < k-1 \leq d-2$:
\begin{align}
(T_0^{\Delta} - Q_1)(T_0^\Delta - Q_2) = 0,
\quad
(T_i^{\Delta} +1)(T_0^{\Delta} - q_{\Delta}) = 0,
\\
(T_0^{\Delta}T_1^{\Delta})^2 = (T_1^{\Delta}T_0^{\Delta})^2,
\quad
T_i^{\Delta}T_{i+1}^{\Delta}T_i^{\Delta} = T_{i+1}^{\Delta}T_i^{\Delta}T_{i+1}^{\Delta},
\quad
T_k^{\Delta} T_j^{\Delta} = T_j^{\Delta} T_k^{\Delta}.
\end{align}
Next we rewrite the setup in {\it. loc. cit.} using the following identifications:
\eq
q_{\Delta}\leftrightarrow q^{-2},
\quad
T_i^{\Delta} \leftrightarrow q\inv T_i.
\endeq
Under the identification, the Jucy-Murphy elements are, for $m \geq 1$,
\eq
\begin{split}
L_m &= (q_{\Delta})^{1-m} T^{\Delta}_{m-1} \ldots T^{\Delta}_{0} \ldots T^{\Delta}_{m-1}
\\
&= (qT^{\Delta}_{m-1}) \ldots (qT^{\Delta}_{0}) \ldots (qT^{\Delta}_{m-1})
\\
&= T_{m-1} \ldots T_{0} \ldots T_{m-1}.
\end{split}
\endeq
Then the cyclotomic relation is 
\eq
(q\inv T_0 - Q_1)(q\inv T_0 - Q_2) = 0,
\quad
\textup{or}
\quad
(T_0 - qQ_1)(T_0 - qQ_2) = 0.
\endeq
This is equivalent to our Hecke relation at the specialization below:
\eq
Q_1 = -q\inv Q,
\quad
Q_2 = q\inv Q\inv.
\endeq
In summary we have the following isomorphism of $K$-algebras.
\begin{prop}
The type B Hecke algebra $\cH^\B_{Q,q}(d)$ is isomorphic to the cyclotomic Hecke algebra $\bbH(2,1,d)$ at the specialization 
$Q_1 = -q\inv Q, Q_2 = q\inv Q\inv$.
\end{prop}

For a composition $\lambda = (\lambda_1, \ldots, \lambda_\ell) \in \NN^\ell$ of $\ell$ parts write
\eq
|\lambda| = \lambda_1 + \ldots + \lambda_\ell,
\quad
\textup{and}
\quad
\ell(\lambda) = \ell.
\endeq
A bicomposition of $d$ is a pair $\lambda = (\lone, \ltwo)$ of compositions such that $|\lone| + |\ltwo| = d$.
We denote the set of bicompositions of $d$ by
$\Lambda_2=\Lambda_2(d)$.
%, and let $\tr$ be the dominance partial order on $\Lambda_d$.
A bicomposition $\lambda$ is a bipartition if $\lone, \ltwo$ are both partitions. The set of bipartitions of $d$ is denoted by $\Lambda^+_2=\Lambda^+_2(d)$.

Following \cite{DJM98b}, the cyclotomic Schur algebras can be defined for any saturated subset $\Lambda$ of the set $\Lambda_2(d)$ of all bicompositions of $d$. 
That is, any subset $\Lambda$ of $\Lambda_2$ satisfying the condition below:
\eq
\textup{If }\mu\in\Lambda, \nu\in\Lambda^+_2(d)
\textup{ and }
\nu \tr \mu,
\textup{ then } \nu \in \Lambda.
\endeq
%In this article we only focus on the following subsets of $\Lambda_2$, for $a+b = n$:
%\eq
%\Lambda_{a,b} = \Lambda_{a,b}(d) = \{\lambda = (\lone, \ltwo) \in \Lambda_2(d) ~|~ \ell(\lone) \leq a, \ell(\ltwo) \leq b\}.
%\endeq 
For each $\Lambda$ we define a cyclotomic Schur algebra
$
\bbS(\Lambda) = \textup{End}_{\bbH}\left(
\bigoplus_{\lambda\in\Lambda} m_\lambda \bbH
\right),
$
where
\eq\label{eq:mld}
m_\lambda = u_{\ell(\lone)}^+ x_\lambda,
\quad
u_{\ell(\lone)}^+ = \prod_{m=1}^{\ell(\lone)} (L_m - Q_2),
\quad
x_\lambda = \sum_{w \in \Sigma_\lambda} T_w,
\endeq
and $\Sigma_\lambda = \Sigma_\lone \times \Sigma_\ltwo$ is the Young subgroup of $\Sigma_d$.
The specialization $S_Q(\Lambda)$ of $\bbS(\Lambda)$ at
$Q_1 = -q\inv Q, Q_2 = q\inv Q\inv$ is then given by
\eq
S_Q(\Lambda)
= \textup{End}_{\cH_{Q,q}^\B}\left(
\bigoplus_{\lambda\in\Lambda} m_\lambda \cH_{Q,q}^\B
\right),
%= \bigoplus_{\lambda, \mu \in\Lambda_{a,b}}
%\textup{Hom}_{\cH_{Q,q}^\B}\left(
% m_\mu \cH_{Q,q}^\B, m_\lambda \cH_{Q,q}^\B
%\right),
\endeq
where 
\eq\label{def:mlambda}
m_\lambda = (L_1 - q\inv Q\inv)\ldots (L_{\ell(\lone)} - q\inv Q\inv) x_\lambda.
\endeq

Let $\cT_0(\lambda, \mu)$ be the set of semi-standard $\lambda$-tableaux of type $\mu$, that is, any $T = (T^{(1)}, T^{(2)}) \in \cT_0(\lambda, \mu)$ satisfies the conditions below:
\begin{enumerate}
\item[(S0)] $T$ is a  $\lambda$-tableau whose entries are ordered pairs $(i,k)$, and the number of $(i,j)$'s appearing is equal to $\mu_i^{(k)}$;
\item[(S1)] entries in each row of each component $T^{(k)}$ are non-decreasing;
\item[(S2)] entries in each column of each component $T^{(k)}$ are strictly increasing;
\item[(S3)] entries in $T^{(2)}$ must be of the form $(i,2)$.
\end{enumerate}
We note that the dimension of the cyclotomic Schur algebra $\Lambda$ is given by
\eq
\dim S_Q(\Lambda) = 
\sum_{\substack{\lambda \in \Lambda_2^+(d)\\ \mu, \nu \in \Lambda}} 
|\cT_0(\lambda, \mu)| \cdot |\cT_0(\lambda, \nu)|.
\endeq
It is then define a ``tensor space'' $T_Q(\Lambda) = \bigoplus_{\lambda\in\Lambda} m_\lambda \cH_{Q,q}^\B$ which has an obvious $S_q(\Lambda)$-$\cH_q^\B(d)$-bimodule structure.
%\rmk
%It is unclear whether the double centralizer property below holds, i.e., whether
%\[
%\textup{End}_{S_Q(\Lambda)}(T(\Lambda)) \simeq \cH^\B_{Q,q}(d).
%\]
%-----------------------------------------------------------------------------------------------
\exa
Let
\eq
\Lambda_{a,b} = \Lambda_{a,b}(d) = \{\lambda = (\lambda^{(1)}, \lambda^{(2)}) \in  \Lambda_2(d)~|~ \ell(\lambda^{(1)}) \leq a, \ell(\lambda^{(2)}) \leq b\}.
\endeq
Recall that the dominance partial order in $\Lambda^+_2(1)$ is given by
$
\mu_2 = (\iyng{1}, \varnothing) \rhd  \mu_1 = (\varnothing, \iyng{1}),
$
and hence $\Lambda_{0,1}(1), \Lambda_{1,1}(1)$ are saturated, while $\Lambda_{1,0}(1)$ is not.
The cardinality of $|\cT_0(\mu_\bullet,\mu_\bullet)|$ is given as below:
\[
|\cT_0(\mu_1, \mu_1)| = 1 = |\cT_0(\mu_2, \mu_1)|  = |\cT_0(\mu_2,\mu_2)|,
\quad
|\cT_0(\mu_1, \mu_2)|  = 0.
\]
Note that $\cT_0(\mu_1, \mu_2)$ is empty since the only $\mu_2$-tableau of type $\mu_1$ is $(\varnothing, \iyoung{12})$, which violates (S3).
Hence, the dimensions of these cyclotomic Schur algebras are
\[
S_q(\Lambda_{0,1}(1)) =1,
\quad
S_q(\Lambda_{0,1}(1))=3.
\]
For $d=2$, the dominance order in $\Lambda^+_2(2)$ is given by
\[
\lambda_5 = (\iyng{2}, \varnothing) \rhd 
\lambda_4 = \left(\iyng{1,1}, \varnothing\right) \rhd 
\lambda_3 = (\iyng{1}, \iyng{1}) \rhd 
\lambda_2 =(\varnothing, \iyng{2}) \rhd 
\lambda_1 =\left( \varnothing, \iyng{1,1}\right).
\]
The sets $\Lambda_{0,2}(2), \Lambda_{1,2}(2)$,  and $\Lambda_{2,2}(2)$ are saturated.
The cardinality of $|\cT_0(\lambda_\bullet,\lambda_\bullet)|$ is given in the following table
\[
\begin{array}{c|ccccc}
\textup{type\textbackslash shape}&\lambda_5&\lambda_4&\lambda_3&\lambda_2&\lambda_1
\\
\hline
\lambda_5&1&0&0&0&0
\\
\lambda_4&1&1&0&0&0
\\
\lambda_3&1&1&1&0&0
\\
\lambda_2&1&0&1&1&0
\\
\lambda_1&1&1&2&1&1
\end{array}
\]
Hence, the dimensions are
\[
\dim S_q(\Lambda_{0,2}(2))=3,
\quad
\dim S_q(\Lambda_{1,2}(2))=7,
\quad
\dim S_q(\Lambda_{2,2}(2))=15.
\]
Recall that $\dim S_q^\B(2,d) = d+1$ for all $d$, hence the algebras $S_q^\B$ and $S_q(\Lambda)$ small ranks do not match in an obvious way.
\endexa
%=============================================
\subsection{Sakamoto-Shoji Algebras}\label{sec:SS}
%=============================================
The cyclotomic Hecke algebra $\bbH(r,1,d)$ does admit a Schur-type duality (cf. \cite{SS99}) with the algebra $U_q(\fgl_{n_1}\times\ldots\times\fgl_{n_r})$ where $n_1 + \ldots + n_r = n$. Hence, 
it specializes to the following double centralizer properties, for $a+b=n$: 
\[
\Ba{{cclccccc}
U_q(\fgl_a \times \fgl_b)
\\
\downarrow
\\
S_q^{\B}(a,b,d) &\crr& T(a,b,d) = (K^n)^{\otimes d}& \crl& \cH^\B_q(d)
}
\]

We will see in \eqref{eq:SST0} that $T_0$ acts as a scalar multiple on $T(a,b,d)$, which is different from our $T_0$-action \eqref{eq:vTi}.
Consquently, the duality is different from the geometric one. We could not locate an identification between $S_q^{\B}(a,b,d)$ and  $S_q(\Lambda)$ for some $\Lambda$ in 
the literature.

Now we set up the compatible version of the cyclotomic Schur duality introduced in \cite{SS99}. 
Let $R' = \QQ(Q, q', u_1, u_2)$,
and let $\bbH_{d,2}$ be the the $R'$-algebra generated by $a_1, \ldots, a_d$ subject to the relations below, for $2\leq i \leq d, 1 \leq j < k-1 \leq d-1$:
\begin{align}
(a_1 - u_1)(a_1 - u_2) = 0,
\quad
(a_i -q')(a_i + (q')\inv) = 0,
\\
(a_1a_2)^2 = (a_2a_1)^2,
\quad
a_i a_{i+1} a_i = a_{i+1} a_i a_{i+1},
\quad
a_k a_j = a_j a_k.
\end{align}
With the identifications below one has the following result. 
\eq
a_i \leftrightarrow T_{i-1},
\quad
q' \leftrightarrow q\inv
\endeq
\begin{prop}
The type B Hecke algebra $\cH^\B_{Q,q}(d)$ is isomorphic to the algebra $\bbH_{d,2}$ at the specialization 
$u_1 = -Q, u_2 =  Q\inv$.
\end{prop}
Let $T_Q(a,b,d) = V_{a,b}^{\otimes d}$ where $V_{a,b}=K^a \oplus K^b$ is the natural representation of $U_q(\fgl_a \times \fgl_b)$ with bases $\{v^{(1)}_1, \ldots, v^{(1)}_a\}$ of $K^a$ and $\{v^{(2)}_1, \ldots, v^{(2)}_b\}$ of $K^b$.
The tensor space $T_Q(a,b,d)$ admits an obvious action of the type A Hecke algebra generated by $T_1, \ldots, T_{d-1}$. The $T_0$-action on $T(a,b,d) $ is more subtle as defined by 
\eq\label{eq:SST0}
T_0 = T_1\inv \circ \ldots \circ T_{d-1}\inv \circ S_{d-1} \circ \ldots \circ S_1 \circ \varpi \in \textup{End}(T(a,b,d)),
\endeq
where $\varpi$ is given by
\eq
\varpi(x_1 \otimes \ldots \otimes x_d) = \begin{cases}
-Q x_1 \otimes \ldots \otimes x_d
&\tif x_1 = v^{(1)}_i \textup{ for some }i;
\\
Q\inv x_1 \otimes \ldots \otimes x_d
&\tif x_1 = v^{(2)}_i \textup{ for some }i,
\end{cases}
\endeq
and that $S_i$ is given by
\eq
S_i(x_1 \otimes \ldots \otimes x_d) = \begin{cases}
T_i (x_1 \otimes \ldots \otimes x_d)
&\tif x_i, x_{i+1} \textup{ both lies in }K^a \textup{ or }K^b;
\\
   \ldots x_{i-1} \otimes x_{i+1} \otimes x_i\otimes x_{i+2} \otimes \ldots
& \textup{otherwise}.
\end{cases}
\endeq
Define 
\eq
S^{\B}_{Q,q}(a,b,d)= \textup{End}_{\cH_{Q,q}^\B(d)}\left(
T_Q(a,b,d)
\right).
\endeq
It is proved in \cite{SS99} that there is a Schur duality as below:
\[
\Ba{{cclccccc}
U_q(\fgl_a \times \fgl_b)
\\
\downarrow
\\
S^{\B}_{q}(a,b,d) &\crr& T(a,b,d)& \crl& \cH_{q}^\B(d)
}
\]
In \cite[Theorem~3.2]{A99} it is also proved an isomorphism theorem under a separation condition on $u_1, u_2$ and $q$. 
Note that the separation condition is equivalent to our invertibility condition at the specialization $u_1 = -Q, u_2 = Q^{-1}$.
\prop
If $f^\B_d(Q, q)$ is invertible in the field $K$, then we have an isomorphism of $K$-algebras:
\eq
S^{\B}_{Q,q}(\nhc, \nhf,d) 
 \to \bigoplus_{i=0}^d S_q^\A(\lceil n / 2 \rceil, i) \otimes S_q^\A(\lfloor n / 2 \rfloor, d - i).
\endeq
As a consequence, $S^{\B}_{Q,q}(\nhc, \nhf,d)$ is isomorphic to the coideal $q$-Schur algebra $S^\B_{Q,q}(n,d)$ under the invertibility condition.
\endprop
In the example below we show that the two algebras do not coincide when the invertibility condition fails.
\exa
Let $a = b = 1, d =2$. Then $T_Q(1,1,2)$ has a basis $\{ v := v^{(1)}_1, w:= v^{(2)}_1\}$. The $T_0$-action is given by
\begin{align*}
&(v\otimes v)T_0 = -Q v\otimes v,
\\
&(v\otimes w)T_0 = -Q v\otimes w,
\\
&(w\otimes v)T_0 = Q\inv (w\otimes v + (q\inv - q) v\otimes w),
\\
&(w\otimes w)T_0 = Q\inv w\otimes w.
\end{align*}
Note that this is essentially different from the $T_0$-action for the coideal Schur algebra given in \eqref{eq:vTi}.

Following \cite[\S2, Example]{A99}, the dimension of $S^\B_{Q,q}(1,1,2)$ is either 3,4 or 5. Note that 10 is excluded since at our specialization $u_1 = -Q, u_2 = Q^{-1}$ it is not possible that $u_1 = u_2 = 0$. 
In contrast, $S^\B_{Q,q}(2,d)$ is always of dimension 3.
\endexa
%=============================================
\subsection{Slim cyclotomic Schur Algebras}\label{sec:slim}
%=============================================
The slim cyclotomic Schur algebra $S_{(u_1,\ldots,u_r)}(n,d)$ introduced in \cite{DDY18} is another attempt to establish a Schur duality for the cyclotomic Hecke algebra $\bbH(r,1,d)$.
When $r=2$, the algebra $S_{(u_1,u_2)}(n,d)$ has the same dimension as the coideal $q$-Schur algebra $S^\B_{Q,q}(2n,d)$; while there is no counterparts for the algebra $S^\B_{Q,q}(2n+1,d)$.

It is conjectured in \cite{DDY18} that there is a weak Schur duality between the cyclotomic Hecke algebras and certain Hopf subalgebras $U_q(\widehat{\fsl}_n)^{(t)}$ of $U_q(\widehat{\fgl}_n)$ for an integer $t$ to be determined.
In our setting it can be phrased as follows:
\[
\Ba{{cclcccccccccc}
U_q(\widehat{\fgl}_n)&\supsetneq&U_q(\widehat{\fsl}_n)^{(t)}
\\
&&\downarrow
\\
&&S^{\widehat{\A}}_q(n,d) &\to& S_{(q,q)}(n,d) &\crr& \Omega^{\otimes d}& \crl& \cH^\B_q(d)
}
\]
Here $S_{(q,q)}(n,d) = \End_{\cH^\B_q(d)}(T_{(q,q)}(n,d))$ is the centralizer algebra of the $\cH^\B_q(d)$-action on a finite dimensional $q$-permutation module $T_{(q,q)}(n,d)$, while $\Omega$ is the (infinite-dimensional) natural representation of $U_q(\widehat{\fgl}_n)$.

We remark that it is called a weak duality in the sense that there are epimorphisms $U_q(\widehat{\fsl}_n)^{(t)} \twoheadrightarrow S_{(q,q)}(n,d)$ and $\cH^\B_q(d) \twoheadrightarrow \End_{S_{(q,q)}(n,d)}(\Omega^{\otimes d})$; while it is not a genuine double centralizer property.

\end{document}